\documentclass[11pt]{amsart}
\usepackage{amsmath,amsthm,amsfonts,amscd,amssymb,enumerate,color,hyperref,latexsym,mathtools,bm}
\usepackage[utf8]{inputenc}
\usepackage[T1,T2A]{fontenc} 
\usepackage[UKenglish]{babel}
\usepackage{graphicx} 
\usepackage{enumitem}
\usepackage{comment}
\usepackage[margin=1in]{geometry}
\usepackage{parskip}
\usepackage{microtype}
\usepackage{tikz-cd}
\hypersetup{colorlinks,allcolors=blue}
\usepackage{url}
\input{xypic}
\usepackage{lineno}
\usepackage{rotating}
\usepackage{multirow}
\usepackage{multicol}
\usepackage{epstopdf}
\usepackage{etoolbox}
\newcounter{dummy}
\numberwithin{equation}{section}
\numberwithin{table}{section}
\usepackage{comment}
\usepackage{doi}
\usepackage{color,soul}
\usepackage{booktabs}
\makeatletter  
\newcommand\myitem[1][]{\item[#1]\refstepcounter{dummy}\def\@currentlabel{#1}}   
\makeatother

\newtheorem{thm}{Theorem}[section]
\newtheorem{cor}[thm]{Corollary}
\newtheorem{lem}[thm]{Lemma}
\newtheorem{prop}[thm]{Proposition}

\theoremstyle{definition}
\newtheorem{defn}[thm]{Definition}

\newtheorem{rem}[thm]{Remark}

\newtheorem*{ack}{Acknowledgements}

\newcommand{\A}{\mathrm{A}}
\newcommand{\T}{\mathrm{T}}
\newcommand{\Wb}{\mathrm{Wb}}
\newcommand{\W}{\mathrm{W}}

\newcommand{\NN}{\mathbb{N}}

\newcommand{\QQ}{\mathbb{Q}}
\newcommand{\RR}{\mathbb{R}}

\newcommand{\HH}{\mathbb{H}}

\newcommand{\MM}{\mathbb{M}}
\renewcommand{\SS}{\mathbb{S}}

\newcommand{\calB}{\mathcal{B}}
\newcommand{\calC}{\mathcal{C}}
\newcommand{\calD}{\mathcal{D}}

\newcommand{\calF}{\mathcal{F}}

\newcommand{\calM}{\mathcal{M}}
\newcommand{\calS}{\mathcal{S}}
\newcommand{\calL}{\mathcal{L}}
\newcommand{\calK}{\mathcal{K}}
\newcommand{\calX}{\mathcal{X}}

\newcommand{\eps}{\varepsilon}

\newcommand{\mset}[1]{\left\{\!\left\{ #1 \right\}\!\right\}}

\usepackage{stmaryrd}

\newcommand{\calP}{\mathcal{P}}

\DeclareMathOperator{\supp}{supp}

\DeclareMathOperator{\dist}{dist}

\DeclareMathOperator{\curv}{curv}

\DeclareMathOperator{\Geo}{Geo}
\DeclareMathOperator{\Opt}{Opt}

\DeclareMathOperator{\proj}{proj}
\DeclareMathOperator{\GeoSel}{GeoSel}
\DeclareMathOperator{\Adm}{Adm}
\DeclareMathOperator{\id}{id}
\DeclareMathOperator{\OptGeo}{OptGeo}

\newcommand{\ggamma}{\boldsymbol{\gamma}}

\newcommand{\aalpha}{\boldsymbol{\alpha}}
\newcommand{\vto}{\overset{v}{\rightharpoonup}}
\newcommand{\wto}{\overset{w}{\rightharpoonup}}
\renewcommand{\angle}{\measuredangle}

\makeatletter
\def\@setthanks{\vspace{-\baselineskip}\def\thanks##1{\@par##1}\thankses}
\makeatother

\title{Optimal partial transport for metric pairs
}

\author[M.~Che]{Mauricio Che}
\address[Che]{Department of Mathematical Sciences, Durham University, United Kingdom.}
\curraddr{Faculty of Mathematics, University of Vienna, Austria.}
\email{mauricio.adrian.che.moguel@univie.ac.at}
\thanks{Supported by CONACYT (Mexico) Doctoral Scholarship No.\ 769708.}

\date{\today}

\begin{document}

\begin{abstract}
In this article we study Figalli and Gigli's formulation of optimal transport between non-negative Radon measures in the setting of metric pairs. We carry over classical characterisations of optimal plans to this setting and prove that the resulting spaces of measures, $\calM_p(X,A)$, are complete, separable and geodesic whenever the underlying space, $X$, is so. We also prove that, for $p>1$, $\calM_p(X,A)$ preserves the property of being non-branching, and for $p=2$ it preserves non-negative curvature in the Alexandrov sense. Finally, we prove isometric embeddings of generalised spaces of persistence diagrams $\calD_p(X,A)$ into the corresponding spaces $\calM_p(X,A)$, generalising a result by Divol and Lacombe. As an application of this framework, we show that several known geometric properties of spaces of persistence diagrams follow from those of $\calM_p(X,A)$, including the fact that $\calD_2(X,A)$ is an Alexandrov space of non-negative curvature whenever $X$ is a proper non-negatively curved Alexandrov space.
\end{abstract}

\maketitle

\section{Introduction}\label{sec:introduction}
Optimal transport provides a geometric perspective on the study of spaces of probability measures. It can be formulated as the problem of minimising the total cost of transferring a given amount of mass between two given distributions, provided we know the cost of delivering mass between any two locations. Remarkably simple though this framework might seem, it has found applications to several different areas and has become a prominent body of research in the intersection of analysis, geometry, and probability (see the encyclopedic presentation in \cite{V09} for a detailed account of the history of this subject). 

An important restriction for classic optimal transport is that it only makes sense for measures with the same total mass, and it is therefore interesting to explore ways to define optimal transport between unbalanced measures. Different approaches have been proposed (see for example \cite{CM10,CPSV18,Figalli10,FG10,Hanin92,LMS18} and references therein), and very recently, Savar\'e and Sodini, in \cite{SS24}, presented a general formulation that includes previous approaches in the case of finite Radon measures.

In \cite{FG10}, Figalli and Gigli introduced a notion of optimal transport between non-negative Radon measures defined on bounded domains in Euclidean space, motivated by finding solutions to evolution equations with Dirichlet boundary conditions, in the spirit of Jordan--Kinderlehrer-Otto scheme \cite{JKO98}. The idea here is to use the boundary of the domain as an unlimited supply and storage of mass, such that it can be used to compensate the difference in mass between the given measures, as long as one pays the cost of transporting it to the boundary (see Section \ref{sec:definitions-OPT} for rigorous definitions). For the sake of simplicity, refer to this notion as {\em optimal partial transport}, although we acknowledge that this terminology has been used in the past to denote the related work by Figalli \cite{Figalli10} (see \cite{ChI15,Indrei13} for further work).

More recently, in \cite{DL21}, Divol and Lacombe revisited and applied this theory to study spaces of persistence diagrams, which arise in topological data analysis (see \cite{CSEH07,ELZ00,ZC05} for foundational results on this area). They extended definitions in \cite{FG10} to Radon measures on unbounded domains in $\RR^n$ and proved several useful properties of the resulting metric spaces (e.g.\ completeness, separability, characterisation of convergence, existence of Fr\'echet means). They also proved that the space of persistence diagrams endowed with the Wasserstein distances can be isometrically embedded into these spaces of measures, by regarding persistence diagrams as discrete measures on a half-plane in $\RR^2$. Also based on the framework of optimal partial transport, Bate and Garcia Pulido proved the embeddability of finite atomic measures, endowed with the optimal partial transport metrics, into Hilbert spaces \cite{BGP24}.

In this article, we extend the theory of optimal partial transport to the setting of {\em metric pairs}. Namely, we consider ordered pairs $(X,A)$, where $X$ is a proper (i.e.\ closed and bounded subsets are compact), complete and separable metric space, and $A$ is a closed, non-empty subset of $X$ (cf.\ \cite{AGC24,CGGGMS2022,CGGGMSV2024}), and non-negative Radon measures $\mu$ on $X\setminus A$ such that the function $\dist(\cdot, A)$ is in $L^p(\mu)$. As a result, we obtain one-parameter families of metric spaces $\{(\calM_p(X,A),\Wb_p)\}_{p\in[1,\infty)}$ associated to each metric pair $(X,A)$. This is a natural extension of Figalli and Gigli's definition, which is recovered when $X = \overline{\Omega}$ and $A=\partial \Omega$, where $\Omega$ is a bounded open set in $\RR^n$. We also recover Divol and Lacombe's setting when $X =\{(x,y)\in\RR^2:x\leq y\}$ and $A=\{(x,y)\in\RR^2:x=y\}$.

The structure of the paper is as follows. We start by proving the existence of optimal partial transport plans and the fact that $\Wb_p$ defines a metric on $\calM_p(X,A)$ (theorems \ref{thm:existence-of-optimal-partial-transport} and \ref{thm:wb-basic-properties}), generalising results in \cite{DL21,FG10}. Along the way, we correct an oversight in the proof of \cite[Proposition 3.2]{DL21} (see discussion in the proof of theorem \ref{thm:existence-of-optimal-partial-transport}). We then prove that the spaces $\calM_p(X,A)$, endowed with the optimal partial transport distance $\Wb_p$, are complete, separable and geodesic, provided that $X$ has the same properties (propositions \ref{prop:competeness and separability} and \ref{prop:geodesics}). Moreover, when $p>1$, we prove that $\calM_p(X,A)$ is non-branching (proposition \ref{prop:non-branching}), and for $p=2$, we prove that $\calM_2(X,A)$ is non-negatively curved in the Alexandrov sense, whenever $X$ is so (theorem \ref{t:alexandrov-OPT}). Finally, in section \ref{sec:embedding-OPT}, we prove that the generalised spaces of persistence diagrams $\calD_p(X,A)$ introduced in \cite{CGGGMS2022} can be isometrically embedded into $\calM_p(X,A)$, generalising \cite[Proposition 3.5]{DL21}. It is important to note that, combined with theorem \ref{t:alexandrov-OPT}, this gives an alternative proof of the fact that $\calD_2(X,A)$ is a non-negatively curved Alexandrov space whenever $X$ is proper and non-negatively curved \cite[Proposition 7.3]{CGGGMS2022} (see also \cite{BH24} for related results), which in the case $X =\{(x,y)\in\RR^2:x\leq y\}$ and $A=\{(x,y)\in\RR^2:x=y\}$ yields the known result, due to Turner, Mileyko, Mukherjee and Harer, that the usual space of persistence diagrams endowed with the $L^2$-Wasserstein distance is non-negatively curved \cite[Theorem 2.5]{TMMH14}.

Bubenik and Elchesen independently developed a theory of optimal transport for metric pairs in \cite{BE2025}. However, they did not address the metric and topological properties that we examine in this paper, such as completeness, separability, the existence of geodesics, the non-branching property, and non-negative curvature in the Alexandrov sense. Therefore, our results complement theirs.

Additionally, a few weeks after we completed the first version of this paper, Erbar and Meglioli published \cite{EM2025}, where they investigate a variational formulation of evolution equations with constant Dirichlet boundary conditions as gradient flows in the spaces $\calM_p(X, A)$, with $X$ being a Euclidean domain and $A$ its boundary. Extending these results to more general metric pairs would be an interesting direction for future research, potentially benefiting from the general framework developed in this paper.

\begin{ack}
I would like to express my gratitude to Fernando Galaz-Garc\'{\i}a for all his support and advise, as this work is part of my doctoral thesis under his supervision. I would also like to thank Jaime Santos-Rodr\'{\i}guez, Martin Kerin, Kohei Suzuki, Alp\'ar M\'esz\'aros,  Amit Einav, Mo Dick Wong, Norbert Peyerimhoff and Mohammad Al Attar for all the valuable comments and fruitful discussions during several sessions of our reading seminar in Durham University,  where the contents of this article were first discussed.  I am also grateful to Javier Casado, Manuel Mellado Cuerno, and Motiejus Valiunas for their feedback on an earlier version of this manuscript. Finally, I thank Ana Luc\'{\i}a Garc\'{\i}a Pulido and Wilhelm Klingenberg for their careful reading of my PhD thesis, which includes this manuscript. I have been financially supported by CONACYT (Mexico), through the Doctoral Scholarship No.\ 769708.
\end{ack}

\section{Preliminaries}

\subsection{Metric geometry}
We briefly recall well-known definitions and results about metric spaces, geodesics, and lower curvature bounds in the Alexandrov sense (see \cite{BBI01} for a more detailed discussion).
\begin{defn}
Let $X$ be a metric space. We denote by $\calC([a,b],X)$ the space of continuous curves $\xi\colon [a,b]\to X$, endowed with the uniform metric. For any $t\in [a,b]$, $e_t\colon \calC([a,b],X)\to X$ is the evaluation map given by $e_t(\xi)=\xi_t=\xi(t)$.  

A {\em constant speed geodesic}, or simply a {\em geodesic}, is a continuous curve $\xi\in \calC([0,1], X)$ such that
\[
d(\xi_s,\xi_t) = d(\xi_0,\xi_1)|s-t|
\]
for any $s,t\in [0,1]$. We denote by $\Geo(X)$ the space of geodesics in $X$, endowed with the uniform metric. We say that $X$ is a {\em geodesic space} if for any $x,y\in X$ there exists $\xi\in\Geo(X)$ such that $\xi_0 = x$ and $\xi_1 = y$.
\end{defn}

It is known that if $X$ is a complete, separable and geodesic space, then $\Geo(X)$, endowed with the uniform metric, is complete and separable. Moreover, if $X$ is a proper space then $\Geo(X)$ is proper as well, by the Arzel\`a-Ascoli theorem.

\begin{defn}\label{def:non branching}
We say that a geodesic space $X$ is \emph{non-branching} if for any $t\in(0,1)$ the map $(e_0,e_t)\colon \Geo(X)\to X\times X$ is injective.
\end{defn}

Alexandrov spaces are synthetic generalisations of Riemannian manifolds with sectional curvature bounded from below. This generalisation comes from the classical Toponogov's comparison theorem in Riemannian geometry (see \cite{GKT68, M89}).

More precisely, the $n$-dimensional \emph{model space} with constant sectional curvature $\kappa$ is given by
\[
\MM^n_\kappa = \left\lbrace
\begin{array}{ll}
\SS^n_\kappa,  &  \mbox{if } \kappa>0,\\
\RR^n,  &  \mbox{if } \kappa=0,\\
\HH^n_\kappa,  &  \mbox{if } \kappa<0,
\end{array}
\right.
\]
where $\SS^n_\kappa$ and $\HH^n_\kappa$ are the sphere and the hyperbolic space with their canonical metrics re-scaled by $1/\sqrt{|\kappa|}$. A \emph{geodesic triangle} $\triangle pqr$ in $X$ consists of three points $p,q,r\in X$ and three minimising geodesics $[pq],\ [qr],\ [rp]$ between those points. A \emph{comparison triangle} in $\MM^2_\kappa$ for $\triangle pqr$ is a geodesic triangle $\widetilde{\triangle}_\kappa pqr = \triangle \widetilde{p}\widetilde{q}\widetilde{r}$ in $\MM^2_\kappa$ such that
\[
d(\widetilde{p},\widetilde{q})=d(p,q),\ d(\widetilde{q},\widetilde{r})=d(q,r),\ d(\widetilde{r},\widetilde{p})=d(r,p).
\]

\begin{defn}\label{def:alex}
We say that $X$ is an \emph{Alexandrov space with curvature bounded below by $\kappa$}, and denote it by $\curv(X)\geq \kappa$, if $X$ is complete, geodesic and satisfies the following condition:
\begin{itemize}
    \item[$(\T_\kappa)$]\label{IT:PROPERTY_T} For any geodesic triangle $\triangle pqr$, any comparison triangle $\widetilde{\triangle}_\kappa pqr$ in $\MM^2_\kappa$ and any point $x\in [qr]$, the corresponding point $\widetilde{x}\in [\widetilde{q}\widetilde{r}]$ such that $d(\widetilde{q},\widetilde{x})=d(q,x)$ satisfies
\begin{align*}
    d(p,x)\geq d(\widetilde{p},\widetilde{x}).
\end{align*}
\end{itemize}
\end{defn}

\begin{rem}\label{rem:non-negative curvature}
Observe that condition $(\T_0)$ can be formulated as follows: for any geodesic triangle $p,q,r\in X$, any geodesic $\xi\in\Geo(X)$ with $\xi_0=q$ and $\xi_1=r$, 
\begin{equation}\label{eq:non-negative curvature}
    d(p,\xi_t)^2 \geq (1-t)d(p,q)^2+td(p,r)^2-(1-t)td(q,r)^2
\end{equation}
holds for any $t\in[0,1]$. This is due to the right hand side of the inequality above being the square of $|\widetilde{p}-\widetilde{\xi_t}|$ in the comparison triangle $\widetilde{\triangle}_0 pqr$.
\end{rem}

\begin{rem}\label{rem:angles and spaces of directions}
Condition $(\T_\kappa)$ is equivalent to the following:
\begin{itemize}
    \item[$(\A_\kappa)$] For any $p\in X$ and any $\xi^1,\xi^2\in \Geo(X)$ such that $\xi^1_0 = \xi^2_0 = p$, the function $(s,t)\mapsto \widetilde\angle_\kappa \xi^1_s p\xi^2_t$ is non-increasing in both $s$ and $t$, where $\widetilde\angle_\kappa \xi^1_s p\xi^2_t$ denotes the angle at $\widetilde{p}$ in the comparison triangle $\widetilde\triangle_\kappa \xi^1_sp\xi^2_t$. 
\end{itemize}
Condition $(\A_\kappa)$ implies that the \emph{angle} between $\xi^1,\xi^2\in\Geo(X)$ with $\xi^1_0=\xi^2_0$, given by
\[
\angle (\xi^1,\xi^2) =\lim_{s,t\to 0} \widetilde\angle_\kappa \xi^1_s p\xi^2_t 
\]
is well-defined. Geodesics that make an angle zero determine an equivalence class called \emph{geodesic direction}. The set of geodesic directions at a point $p\in X$ is denoted by $\Sigma'_p$. When equipped with the angle  metric $\angle$, the set $\Sigma_p'$ is a metric space. The completion of  $\left(\Sigma_p', \angle\right)$  is called the \emph{space of directions of $X$ at $p$}, and is denoted by $\Sigma_p$. Note that in a closed Riemannian manifold the space of directions at any point is isometric to the unit sphere in the tangent space to the manifold at the given point. 
\end{rem}

\subsection{Optimal transport and measure theory}
We now recall definitions and classical results from optimal transport (see \cite{AG13,S15,V09} for detailed expositions).

Let $X$ be a metric space and $\calP(X)$ be the set of Borel probability measures on $X$. The {\em support} of a Borel measure $\mu$ on $X$, denoted by $\supp(\mu)$, is the smallest closed set $E\subset X$ such that $\mu(X\setminus E) = 0$. For any Borel map $T\colon X\to Y$ between metric spaces, the \emph{push-forward map} $T_\#\colon \calP(X)\to \calP(Y)$ is given by 
\[T_\#\mu(E) = \mu(T^{-1}(E))\] 
for any Borel set $E\subset Y$. 

Let us recall the following Borel measurable selection principle (see, for example, \cite[Theorem 1]{A74}).

\begin{thm}\label{t:azoff}
Let $X$ and $Y$ be complete and separable metric spaces, and $E$ a closed, $\sigma$-compact (i.e. $E$ can be covered with countably many compact sets) subset of $X\times Y$. If $\pi^1\colon X\times Y\to X$ is the projection onto the first factor, then $\pi^1(E)$ is a Borel set in $X$ and there exists a Borel measurable map $\phi\colon \pi^1(E)\to Y$ whose graph is contained in $E$.
\end{thm}

\begin{defn}\label{def:optimal transport}
Given $\mu,\nu\in\calP(X)$, we say that $\gamma\in\calP(X\times X)$ is a \emph{transport plan} between $\mu$ and $\nu$ if $\pi^1_\#\gamma = \mu$ and $\pi^2_\#\gamma=\nu$, where $\pi^1,\ \pi^2\colon X\times X\to X$ are the coordinate maps. The set of transport plans between $\mu$ and $\nu$ is denoted by $\Adm(\mu,\nu)$. For any $p\in[1,\infty)$, the \emph{$L^p$-Wasserstein metric} between $\mu,\nu\in\calP(X)$ is defined as
\begin{equation}\label{eq:wasserstein}
\W_p(\mu,\nu)=\inf_{\gamma\in\Adm(\mu,\nu)}\left(\int d(x,y)^p\ d\gamma(x,y)\right)^{1/p}.
\end{equation}
Any minimiser $\gamma$ for \eqref{eq:wasserstein} is an {\em optimal plan} between $\mu$ and $\nu$. The set of optimal plans between $\mu$ and $\nu$ is denoted by $\Opt(\mu,\nu)$.
\end{defn}

\begin{thm}\label{thm:classical optimal transport}
Let $X$ be a complete and separable metric space, and $\mu,\nu\in\calP(X)$ such that
\begin{equation}\label{eq:finite p-moment}
\int_X d(x,x_0)^p\ d\mu(x), \int_X d(y,x_0)^p\ d\nu(y) < \infty
\end{equation}
for some (and therefore any) $x_0\in X$. Then $\Opt(\mu,\nu)\neq \varnothing$. Moreover, $\W_p$ defines a metric in $\calP_p(X)$, the set of measures in $\calP(X)$ satisfying \eqref{eq:finite p-moment}.  
\end{thm}

The following result, commonly known in the optimal transport jargon as the {\em gluing lemma}, plays a role in the proof of theorem \ref{thm:classical optimal transport}, and will be useful later on.

\begin{thm}\label{thm:classical gluing}
Let $\mu^1,\mu^2,\mu^3\in \calP(X)$, $\gamma^{12}\in \Adm(\mu^1,\mu^2)$, and $\gamma^{23}\in \Adm(\mu^2,\mu^3)$. Then there exists $\gamma^{123}\in \calP(X\times X\times X)$ such that
\begin{align*}
\pi^{12}_\#\gamma^{123}=\gamma^{12},\\
\pi^{23}_\#\gamma^{123}=\gamma^{23}.
\end{align*}

More generally, if $\Gamma^1\in\calP(\calX^1)$, $\Gamma^2\in\calP(\calX^2)$, and $F^i\colon \calX^i \to \calX$, $i=1,2$, are measurable maps such that $F^1_\#\Gamma^1=F^2_\#\Gamma^2$, then there exists 
\[
\hat{\Gamma}\in\calP(\{(x_1,x_2)\in \calX^1\times\calX^2 : F^1(x_1)=F^2(x_2)\})
\]
such that $\pi^i_\#\hat{\Gamma}= \Gamma^i$, $i=1,2$. 
\end{thm}

It is also possible to characterise optimal plans in terms of cyclical monotonicity and the existence of Kantorovich potentials.

\begin{defn}\label{def:cyclical monotonicity and c-concavity}
We say that a set $\Gamma\subset X\times X$ is {\em $c$-cyclically monotone} if,  for any $n\in\NN$, any $\{(x_i,y_i)\}_{i=1}^{n}\subset \Gamma$ and any permutation $\sigma$ of $\{1,\dots,n\}$,
\[
\sum_{i=1}^{n} d(x_i,y_i)^p\leq \sum_{i=1}^{n} d(x_i,y_{\sigma(i)})^p
\] holds.
The {\em $c$-transform} of a function $\phi\colon X\to\RR\cup\{-\infty\}$ is given by
\[
\phi^c(y) = \inf_{x\in X} \{d(x,y)^p - \phi(x)\}.
\]
We say that a function $\phi\colon X\to \RR\cup\{-\infty\}$ is {\em $c$-concave} if there exists a function $\psi\colon X\to \RR\cup\{-\infty\}$ such that $\phi(x) = \psi^c(x)$.

The {\em $c$-superdifferential} of a $c$-concave function $\phi$ is the set
\[
\partial^c_+\phi = \{(x,y)\in X\times X: c(x,y) = \phi(x)+\phi^c(y)\}.
\]
\end{defn}

\begin{thm}\label{thm:classical criteria optimal transport}
Let $X$ be a complete and separable metric space, $\mu,\nu\in\calP_p(X)$, and $\gamma\in\Adm(\mu,\nu)$. Then the following conditions are equivalent:
\begin{enumerate}
    \item $\gamma\in\Opt(\mu,\nu)$;
    \item $\supp(\gamma)$ is $c$-cyclically monotone;
    \item There is a $c$-concave function $\phi$ (known as a {\em Kantorovich potential} of $\gamma$) with $\max\{0,\phi\}\in L^1(\mu)$ and $\supp(\gamma)\subset \partial^+_c\phi$.
\end{enumerate}
\end{thm}

We will consider a couple of different notions of convergence of measures.

\begin{defn}\label{d:weak convergence}
A sequence $\{\mu\}_{n\in\NN}$ of bounded Borel measures on $X$ is {\em weakly convergent} to $\mu$, and we write $\mu_n\wto\mu$, if 
\[
\int_X f\ d\mu_n\to \int_X f\ d\mu\quad \text{for any}\ f\in\calC_b(X),
\]
where $\calC_b(X)$ is the set of continuous and bounded functions on $X$.
\end{defn}

Lemma \ref{lem:prokhorov} below is a particular case of the classical Prokhorov's theorem (see \cite[Theorem 4.2]{Kallenberg17}). 

\begin{lem}\label{lem:prokhorov}
    Let $\calF$ be a set of non-negative Borel measures with finite mass (i.e.\ $\mu_n(X)<\infty$ for all $n\in\NN$) on a complete and separable metric space $X$. Then $\calF$ is weakly precompact (i.e.\ every sequence in $\calF$ has a weakly convergent subsequence) if and only if the following conditions hold:
    \begin{enumerate}
        \item\label{bounded total variation} $\calF$ has {\em uniformly bounded total variation}, i.e.\ $\sup_{\mu\in\calF} \mu(X) < \infty$,
        \item\label{tightness} $\calF$ is {\em tight}, i.e.\ for any $\eps>0$ there exist a compact set $K\subset X$ such that $\sup_{\mu\in\calF}\mu(X\setminus K)<\eps$.
    \end{enumerate}
\end{lem}

An even weaker notion of convergence is that of vague convergence of Radon measures.

\begin{defn}\label{def:radon}
A {\em Radon measure} $\mu$ on $X$ is a Borel measure that is both \textit{finite on compact sets} (i.e.\ $\mu(K)<\infty$ for any compact set $K\subset X$) and \textit{inner regular} (i.e.\ for any Borel measurable set $E\subset X$, $\mu(E)$ can be approximated from below by the measures of compact sets $K\subset E$). We denote by $\calM(X)$ the set of Radon measures on $X$. 
\end{defn}

\begin{rem}\label{rem:locally finite equals radon}
If $X$ is a separable and locally compact metric space, any Borel measure on $X$ that is finite on compact sets is also inner regular (see, for example, \cite[Theorem 7.8]{Folland1999}). Since we only consider proper metric spaces (which, in particular, are separable and locally compact), we will only need to verify finiteness on compact sets in order to prove that a Borel measure is Radon. 
\end{rem}

\begin{defn}\label{d:vague convergence}
A sequence $\{\mu_n\}_{n\in\NN}$ of Radon measures on $X$ is \textit{vaguely convergent} to $\mu$, and we write $\mu_n\vto\mu$, if 
\begin{equation}\label{eq:vague-convergence}
\int_X f\ d\mu_n\to \int_X f\ d\mu\quad \text{for any $f\in \calC_c(X)$,}
\end{equation}
where $\calC_c(X)$ is the set of compactly supported continuous functions on $X$.
\end{defn}

Lemmas \ref{lem:vague-relative-compactness} and \ref{lem:vague-convergence} below are adaptations of known results about the vague topology in the setting of separable, locally compact metric spaces (see \cite[Section 15.7]{K86} for details). In particular, these results are applicable to proper metric spaces.

\begin{lem}\label{lem:vague-relative-compactness}
Let $X$ be a separable, locally compact metric space, and let $\calF\subset \calM(X)$. Then $\calF$ is vaguely relatively compact if and only if
\[
\sup\{\gamma(K):\gamma\in \calF\} <\infty
\]
for any compact $K\subset X$.
\end{lem}

\begin{lem}\label{lem:vague-convergence}
Let $X$ be a separable, locally compact metric space, and let $\{\gamma_n\}_{n\in\NN}\subset \calM(X)$. Then the following are equivalent:
\begin{enumerate}
    \item $\gamma_n\vto \gamma$ for some $\gamma\in \calM(X)$.
    \item For any bounded open $U\subset X$ and any bounded closed $F\subset X$,
    \[
    \gamma(U) \leq \liminf_{n\to\infty} \gamma_n(U)
    \]
    and
    \[
    \gamma(F) \geq \limsup_{n\to\infty} \gamma_n(F).
    \]
\end{enumerate}
\end{lem}

\section{Optimal partial transport for metric pairs}\label{sec:definitions-OPT}

A \emph{metric pair} is an ordered pair $(X,A)$ where $X$ is a metric space and $A\subseteq X$ is closed and non-empty. Moreover, we will assume throughout the article that $X$ is complete, separable and proper.

We denote by $\calM(X)$ the set of non-negative {\em Radon measures} on $X$, i.e.\ inner regular Borel measures that are finite on compact sets, and endow it with the vague topology. Then, given a metric pair $(X,A)$ and $p\in [1,\infty)$, we define 
\[
\calM_p(X,A) = \left\{\mu\in \calM(\Omega): \int_\Omega d(x,A)^p\ d\mu(x) < \infty\right\},
\]
where $\Omega = X\setminus A$, and for any $\mu,\nu \in \calM_p(X,A)$, we define the {\em $L^p$-optimal partial transport metric} by
\begin{equation}\label{e:wb}
\Wb_p(\mu,\nu) = \inf_{\gamma\in \Adm(\mu,\nu)} \left(\int_{E_\Omega} d(x,y)^p\ d\gamma(x,y)\right)^{1/p},
\end{equation}
where $\Adm(\mu,\nu)$ is the set of non-negative Borel measures $\gamma$ on $E_\Omega = X\times X\setminus A\times A$ whose marginals restricted to $\Omega$ are $\mu$ and $\nu$, i.e.\ the following equations hold
\begin{equation}\label{e:admissible1}
\pi^1_\#\gamma|_\Omega = \mu,\quad \pi^2_\#\gamma|_\Omega = \nu.
\end{equation}

\begin{rem}\label{rem:partial transport plans are radon}
    Observe that condition \eqref{e:admissible1} implies that $\gamma\in \calM(E_\Omega)$. Indeed, for any compact $K\subset E_\Omega$ there are compact sets $K', K''\subset \Omega$ such that $K\subset K'\times X\cup X\times K''$, therefore
    \[
    \gamma(K) \leq \gamma(K'\times X) + \gamma(X\times K'') = \mu(K')+\nu(K'')<\infty,
    \]
    since $\mu,\nu\in\calM(\Omega)$.  Moreover, since $E_\Omega$ is an open subset of the separable and locally compact metric space $X\times X$, it is separable and locally compact itself, and the claim follows from Remark \ref{rem:locally finite equals radon}. Furthermore, it follows that $\Adm(\mu,\nu)$ is a vaguely relatively compact subset of $\calM(E_\Omega)$, by lemma \ref{lem:vague-relative-compactness}.
\end{rem}

\begin{rem}\label{rem:support in diagonal of AxA}
Given $\gamma \in \Adm(\mu,\nu)$, we define
\[
C(\gamma) = \int_{E_\Omega} d(x,y)^p\ d\gamma(x,y)
\]
and denote $\gamma_R^S = \gamma|_{R\times S}$, for any $R,S\subset X$ such that $R\times S \subset E_\Omega$. In particular, 
\[
\gamma = \gamma_\Omega^\Omega + \gamma_\Omega^A + \gamma_A^\Omega.
\]
\end{rem}

\begin{rem}
Regarding the terminology of {\em partial} transport plans, it comes from the fact that, whenever $\gamma$ satisfies \eqref{e:admissible1}, the measure $\gamma_\Omega^\Omega$ can be regarded as a classical transport plan between the measures $\widetilde\mu = \pi^1_\#(\gamma^{\Omega}_{\Omega})$ and $\widetilde\nu = \pi^1_\#(\gamma^{\Omega}_{\Omega})$, which satisfy that $\widetilde\mu\leq\mu$ and $\widetilde\nu\leq\nu$, therefore partially transporting mass between $\mu$ and $\nu$.
\end{rem}

As a consequence of the Borel measurable selection principle (theorem \ref{t:azoff}), we obtain the following lemma, which will be useful in the sequel.

\begin{lem}\label{lem:projection onto A}
    Let $(X,A)$ be a metric pair. Then there exists a Borel measurable map $\proj_A\colon X\to A$ such that $d(x,\proj_A(x))= d(x,A)$ for all $x\in X$.
\end{lem}
\begin{proof}
Since $A$ is closed and $X$ is complete and separable, it follows that $A$ is complete and separable endowed with the restricted metric. Moreover, since $X$ is proper and $d$ is continuous, the set
\[
E = \{(x,y)\in X\times A: d(x,y)=d(x,A)\}
\]
is $\sigma$-compact, closed, and $\pi^1(E)=X$. Therefore, by theorem \ref{t:azoff}, the claim follows.
\end{proof}

\begin{rem}
Observe that $\Wb_p(\mu,\nu)$ is well-defined, non-negative, and finite. Indeed, let $\mu,\nu\in\calM_p(X,A)$. Then a straightforward computation shows that the measure
\[
\gamma = (\id,\proj_A)_\# \mu + (\proj_A,\id)_\# \nu
\]
is in $\Adm(\mu,\nu)$, and it satisfies 
\[
0\leq C(\gamma) = \int_\Omega d(x,A)^p\ d\mu(x) + \int_\Omega d(y,A)^p\ d\nu(y) < \infty.
\]
Furthermore, the zero measure belongs to $\calM_p(X,A)$, and due to lemma \ref{lem:projection onto A}, for any $\mu\in \calM_p(X,A)$ we have 
\begin{equation}\label{eq:distance-to-zero}
\Wb_p^p(\mu,0) = \int_X d(x,A)^p\ d\mu(x).
\end{equation}
Indeed, the partial transport plan $\gamma=(\id,\proj_A)_\#\mu\in \Adm(\mu,0)$ satisfies
\[
C(\gamma) = \int_{\Omega} d(x,\proj_A(x))^p\ d\mu(x) = \int_{\Omega} d(x,A)^p\ d\mu(x).
\] 
On the other hand, for any $\widetilde\gamma\in \Adm(\mu,0)$ we have that $\pi^2_\#\widetilde\gamma|_\Omega = 0$, which implies that $\widetilde\gamma(X\times \Omega) = 0$. Therefore we have
\begin{align*}
C(\widetilde\gamma) &= \int_{E_\Omega} d(x,y)^p\ d\widetilde\gamma(x,y) \\
&= \int_{\Omega\times A} d(x,y)^p\ d\widetilde\gamma(x,y)\\
&\geq \int_{\Omega\times A} d(x,A)^p\ d\widetilde\gamma(x,y) \\
&= \int_\Omega d(x,A)^p\ d\mu(x).
\end{align*}
This proves equation \eqref{eq:distance-to-zero}.
\end{rem}

The following lemma is a natural generalisation of the well-known gluing lemma for probability measures. 

\begin{lem}\label{lem:gluing}
Let $\mu^1,\mu^2,\mu^3\in \calM_p(X,A)$, $\gamma^{12}\in \Adm(\mu^1,\mu^2)$, and $\gamma^{23}\in \Adm(\mu^2,\mu^3)$. Then there exists a Borel measure $\gamma^{123}$ on $X\times X\times X$ such that
\begin{align*}
\pi^{12}_\#\gamma^{123}=\gamma^{12}+\sigma^{12},\\
\pi^{23}_\#\gamma^{123}=\gamma^{23}+\sigma^{23},
\end{align*}
where $\sigma^{12}$ and $\sigma^{23}$ are Borel measures supported on $\Delta(A\times A)$.    
\end{lem}

\begin{proof}
From the hypothesis, we can see that
\[
\mu^2 = \pi^2_\#\gamma^{12}|_{\Omega} = \pi^2_\#((\gamma^{12})_{X}^{\Omega}+ (\gamma^{12})_{\Omega}^{A})|_{\Omega}= \pi^2_\#(\gamma^{12})_{X}^{\Omega}.
\]
Analogously, $\mu^2 = \pi^1_\#(\gamma^{23})_{\Omega}^{X}$. Therefore, by applying the classical gluing lemma, since $\mu^2$ is a Radon measure, we can find a measure $\widetilde{\gamma}^{123}$ on $X\times \Omega\times X$ such that 
\begin{align*}
    \pi_\#^{12}\widetilde\gamma^{123} = (\gamma^{12})_{X}^{\Omega},\\
    \pi_\#^{23}\widetilde\gamma^{123}= (\gamma^{23})_{\Omega}^{X}.
\end{align*}
We define
\begin{align*}
    \widetilde\sigma^{12} &= (\pi^1,\pi^1,\pi^2)_\#(\gamma^{23})^{\Omega}_{A},\\
    \widetilde\sigma^{23} &= (\pi^1,\pi^2,\pi^2)_\#(\gamma^{12})_{\Omega}^{A},\\
    \gamma^{123} &= \widetilde\gamma^{123} + \widetilde\sigma^{12} + \widetilde\sigma^{23},\\
    \sigma^{12} &= (\pi^1,\pi^1)_\#(\gamma^{23})^{\Omega}_{A},\\
    \sigma^{23} &= (\pi^2,\pi^2)_\#(\gamma^{12})_{\Omega}^{A}.
\end{align*}
We can check that these measures work. Indeed,
\begin{align*}
    \pi^{12}_\#\gamma^{123} &= \pi^{12}_\#\widetilde\gamma^{123} + \pi^{12}_\#\widetilde\sigma^{12} + \pi^{12}_\#\widetilde\sigma^{23}\\
    &= (\gamma^{12})_{X}^{\Omega} + \pi^{12}_\#(\pi^1,\pi^1,\pi^2)_\#(\gamma^{23})^{\Omega}_{A} + \pi^{12}_\#(\pi^1,\pi^2,\pi^2)_\#(\gamma^{12})_{\Omega}^{A}\\
    &= (\gamma^{12})_{X}^{\Omega} + (\pi^1,\pi^1)_\#(\gamma^{23})^{\Omega}_{A} + (\pi^1,\pi^2)_\#(\gamma^{12})_{\Omega}^{A}\\
    &= (\gamma^{12})_{X}^{\Omega} + \sigma^{12} + (\gamma^{12})_{\Omega}^{A}\\
    &= \gamma^{12} + \sigma^{12}.
\end{align*}
Similarly with $\pi^{23}_\#\gamma^{123}$.
\end{proof}

The following result guarantees the existence of optimal partial transport plans in the setting of proper metric spaces, generalising \cite[p.\ 4]{FG10} and \cite[Proposition 3.2]{DL21}.

\begin{thm}\label{thm:existence-of-optimal-partial-transport}
Let $(X,A)$ be a metric pair and $p\in [1,\infty)$. Then, for any $\mu,\nu\in\calM_p(X,A)$, the set
\[
\Opt(\mu,\nu) = \{\gamma\in \Adm(\mu,\nu) : C(\gamma) = \Wb^p_p(\mu,\nu)\}
\]
is non-empty. Moreover, the set $\Opt(\mu,\nu)$ is vaguely compact.
\end{thm}

We will need the following technical lemma for the proof of theorem \ref{thm:existence-of-optimal-partial-transport}, and later on, for the proof of theorem \ref{thm:wb-basic-properties}. The proof of this lemma follows ideas from \cite[Footnotes in Sections 2.1 and 2.2]{AG13}.

\begin{lem}\label{lem:weak convergence of minimising sequence}
Let $\mu,\nu\in\calM_p(X,A)$, $\eps>0$, and let $\Opt_\eps(\mu,\nu)$ be the set of $\gamma \in \Adm(\mu,\nu)$ such that 
\begin{equation}\label{eq:gamma-eps}
\Wb_p^p(\mu,\nu)\leq C(\gamma)\leq \Wb_p^p(\mu,\nu)+\eps.
\end{equation}
Then, for any compact $C\subset \Omega$, the set $\{\gamma_{C}^{X}:\gamma\in\Opt_\eps(\mu,\nu)\}$ is weakly relatively compact.
\end{lem}

\begin{proof}
Take $p_0\in X$ and $0<r<R$ such that $C\subset \overline{B}_r(p_0)\subset \overline{B}_R(p_0)$. By H\"older's inequality, for any $\gamma\in\Adm(\mu,\nu)$, we have
\begin{multline}\label{eq:tight1}
    \left|\int_{C\times (X\setminus \overline{B}_R(p_0))}d(y,p_0)-d(x,p_0)\ d\gamma(x,y)\right|^p\\ \leq \gamma(C\times (X\setminus \overline{B}_R(p_0)))^{p-1}\int_{C\times (X\setminus \overline{B}_R(p_0))} d(x,y)^p\ d\gamma(x,y),
\end{multline}
whereas
\begin{equation}\label{eq:tight2}
    \left|\int_{C\times (X\setminus \overline{B}_R(p_0))}d(y,p_0)-d(x,p_0)\ d\gamma(x,y)\right|^p \geq (R-r)^p\gamma(C\times (X\setminus \overline{B}_R(p_0)))^p
\end{equation}
since $\dist(\overline{B}_r(p_0),X\setminus \overline{B}_R(p_0)) \geq R-r$. Combining \eqref{eq:gamma-eps}, \eqref{eq:tight1}, 
 and \eqref{eq:tight2}, we get
\[
\gamma(C\times (X\setminus \overline{B}_R(p_0))) \leq \dfrac{\Wb_p^p(\mu,\nu)+\eps}{(R-r)^p},
\]
for any $\gamma\in\Opt_\eps(\mu,\nu)$, which can be made arbitrarily small by fixing $r$ and letting $R$ tend to infinity. This proves that $\{\gamma_{C}^{X}:\gamma\in\Opt_\eps(\mu,\nu)\}$ is tight. Moreover, since $\Opt_\eps(\mu,\nu)\subset\Adm(\mu,\nu)$, each $\gamma_C^X$ has total mass $\mu(C)<\infty$, therefore $\{\gamma_{C}^{X}:\gamma\in\Opt_\eps(\mu,\nu)\}$ has uniformly bounded total variation. By theorem \ref{lem:prokhorov}, the claim follows.
\end{proof}

As a consequence of the previous lemma, we obtain that the sets $\Opt_\eps(\mu,\nu)$ are closed with respect to the vague topology. The proof follows along the same lines of those of \cite[Proposition 3.2]{DL21} and arguments in \cite[p.\ 4]{FG10}. Observe that the chain of equations in \eqref{eq:equations of divol-lacombe} below is the same as one in the proof of \cite[Proposition 3.2]{DL21}, but we need lemma \ref{lem:weak convergence of minimising sequence} to justify it, even in the Euclidean setting. This is because, even when $f\in \calC_c(\Omega)$, the composition $f\circ \pi^1\colon \Omega\times X\to \RR$ is not of compact support when $X$ is not compact.
\begin{cor}\label{lem:gamma-eps is closed}
Let $\mu,\nu\in\calM_p(X,A)$, $\eps>0$, and let $\Opt_\eps(\mu,\nu)$ be defined as in lemma \ref{lem:weak convergence of minimising sequence}. Then $\Opt_\eps(\mu,\nu)$ is vaguely closed.
\end{cor}
\begin{proof}
    Let $\{\gamma_k\}_{k\in\NN}\subset \Opt_\eps(\mu,\nu)$ be such that $\gamma_k\vto \gamma$ for some $\gamma\in \calM(E_\Omega)$. We first observe that $\gamma\in\Adm(\mu,\nu)$. Indeed, if $f\in\calC_c(\Omega)$ then, by applying lemma \ref{lem:weak convergence of minimising sequence} with $C=\supp(f)$, we can assume that, up to passing to a subsequence, $\left\{(\gamma_k)_{\supp(f)}^{X}\right\}_{k\in \NN}$ is weakly convergent to $\gamma_{\supp(f)}^{X}$. Therefore,
    \begin{equation}\label{eq:equations of divol-lacombe}
        \int_{\Omega} f\ d\pi^1_\#\gamma = \int_{\Omega\times X} f\circ\pi^1\ d\gamma=\lim_{k\to \infty} \int_{\Omega\times X} f \circ \pi^1\ d\gamma_k=\int_{\Omega} f\ d\mu,
    \end{equation}
    where the second equality follows from the fact that $f\circ \pi^1\in\calC_b(\Omega\times X)$. This implies that $\pi^1_\#\gamma|_{\Omega} = \mu$, and analogously we obtain $\pi^2_\#\gamma|_{\Omega} = \nu$.
    
    Now, we prove that $\gamma$ satisfies \eqref{eq:gamma-eps}. Indeed, by lemma \ref{lem:vague-convergence} applied to the sequence $\{d(\cdot,\cdot)^p \gamma_{k}\}_{k\in\NN}$, which is vaguely convergent to $d(\cdot,\cdot)^p\gamma$, and any bounded open set $U\subset E_\Omega$, we get
    \[
    C(\gamma|_U) \leq \liminf_{k\to \infty} C((\gamma_{k})|_U)\leq \Wb_p^p(\mu,\nu)+\eps.
    \]
    By the monotone convergence theorem, 
    \[
    C(\gamma)\leq \Wb_p^p(\mu,\nu)+\eps,
    \]
    and the claim follows.
\end{proof}

\begin{proof}[Proof of theorem \ref{thm:existence-of-optimal-partial-transport}]
Let $\mu,\nu\in\calM_p(X,A)$ and, for every $k\in\NN$, let $\gamma_k\in\Opt_{1/k}(\mu,\nu)$, that is $\gamma_k\in \Adm(\mu,\nu)$ and
\begin{equation}\label{eq:minimisers-opt}
    \Wb_p^p(\mu,\nu) \leq C(\gamma_k)\leq \Wb_p^p(\mu,\nu)+1/k.
\end{equation}
Remark \ref{rem:partial transport plans are radon} yields that $\{\gamma_k\}_{k\in\NN}$ is a vaguely relatively compact subset of $\calM(E_\Omega)$. Consequently, up to passing to a subsequence, we can assume that there exists $\gamma\in\calM(E_\Omega)$ such that $\gamma_k\vto \gamma$. Corollary \ref{lem:gamma-eps is closed} implies that $\gamma\in \Opt_{1/k}(\mu,\nu)$ for any $k\in\NN$, which means that
\[
\Wb_p^p(\mu,\nu)\leq C(\gamma)\leq\Wb_p^p(\mu,\nu)+\frac{1}{k}
\]
for all $k\in\NN$. Therefore, $\gamma\in\Opt(\mu,\nu)$.

Regarding the second part of the theorem, observe that if $\{\gamma_k\}_{k\in\NN}\subset \Opt(\mu,\nu)$ then, by the previous arguments, up to passing to a subsequence, $\gamma_k\vto \gamma$ for some $\gamma\in\calM(E_\Omega)$. The fact that $\gamma\in\Opt(\mu,\nu)$ also follows from the arguments above.
\end{proof}

We now prove that $\Wb_p$ is a metric and that it is vaguely lower semi-continuous. This statement, and its proof, are adaptations of \cite[Theorem 2.2]{FG10} to the setting of proper metric spaces.

\begin{thm}\label{thm:wb-basic-properties}
Let $(X,A)$ be a metric pair and $p\in [1,\infty)$. Then the function $\Wb_p$ is a metric on $\calM_p(X,A)$. Moreover, $\Wb_p$ is lower semi-continuous with respect to the vague topology.
\end{thm}

\begin{proof}
Let $\mu,\nu\in\calM_p(X,A)$. Then the measure
\[
\gamma = (\id,\proj_A)_\# \mu + (\proj_A,\id)_\# \nu
\]
is in $\Adm(\mu,\nu)$, and it satisfies 
\[
C(\gamma) = \int_\Omega d(x,A)^p\ d\mu(x) + \int_\Omega d(y,A)^p\ d\nu(y) < \infty.
\]
It is also clear that $\Wb_p$ is non-negative and symmetric. Moreover, due to theorem \ref{thm:existence-of-optimal-partial-transport}, $\Wb_p(\mu,\nu)=0$ if and only if $C(\gamma)=0$ for some $\gamma\in \Adm(\mu,\nu)$, which is equivalent to
\[
\int_{E_\Omega} d(x,y)^p\ d\gamma(x,y) = 0.
\]
However, this is the same as $\supp(\gamma) \subset \Delta(E_\Omega)$, which in turn is equivalent to $\pi^1_\#\gamma=\pi^2_\#\gamma$. The latter implies that $\mu=\nu$. Conversely, if $\mu=\nu$ then we can take $\gamma=(\id,\id)_\#\mu\in \Adm(\mu,\nu)$, which satisfies $C(\gamma)=0$.

For the triangle inequality, we need lemma \ref{lem:gluing}. Indeed, let $\mu^1,\mu^2,\mu^3\in\calM_p(X,A)$ and choose $\gamma^{12}\in\Opt(\mu^1,\mu^2)$ and $\gamma^{23}\in\Opt(\mu^2,\mu^3)$ (which we can do thanks to proposition \ref{thm:existence-of-optimal-partial-transport}). By lemma \ref{lem:gluing}, there is a measure $\gamma^{123}$ on $X\times X\times X$ such that 
\begin{align*}
\pi^{12}_\#\gamma^{123}=\gamma^{12}+\sigma^{12},\\ \pi^{23}_\#\gamma^{123}=\gamma^{23}+\sigma^{23},
\end{align*}
with $\supp(\sigma^{12}),\ \supp(\sigma^{23})\subset \Delta(A\times A)$. In particular, $\pi^1_\#\gamma^{123}|_\Omega = \mu^1$ and $\pi^3_\#\gamma^{123}|_\Omega = \mu^3$, therefore $\pi^{13}_\#\gamma^{123}|_{E_\Omega}\in \Adm(\mu^1,\mu^3)$. This implies
\begin{align*}
    \Wb_p(\mu^1,\mu^3) &\leq C(\pi^{13}_\#\gamma^{123}|_{E_\Omega})^{1/p}\\
    &\leq \left\|d\circ \pi^{13}\right\|_{L^p(\gamma^{123})}\\
    &\leq \left\|d\circ \pi^{12}+d\circ \pi^{23}\right\|_{L^p(\gamma^{123})}\\
    &\leq \left\|d\circ \pi^{12}\right\|_{L^p(\gamma^{123})}+\left\|d\circ \pi^{23}\right\|_{L^p(\gamma^{123})}\\
    &= \left\|d\right\|_{L^p(\gamma^{12}+\sigma^{12})}+ \left\|d\right\|_{L^p(\gamma^{23}+\sigma^{23})}\\
    &= C(\gamma^{12})^{1/p} + C(\gamma^{23})^{1/p}\\
    &= \Wb_p(\mu^1,\mu^2)+\Wb_p(\mu^2,\mu^3).
\end{align*}

To prove that $\Wb_p$ is lower semi-continuous with respect to the vague topology, let $\mu_n\vto\mu$ and $\nu_n\vto \nu$. If $\liminf_{n\to \infty} \Wb_p(\mu_n,\nu_n)=\infty$, the result follows trivially. Otherwise, up to passing to a subsequence, we can assume that
\[
\liminf_{n\to \infty} \Wb_p(\mu_n,\nu_n) = \lim_{n\to \infty} \Wb_p(\mu_n,\nu_n).
\]
For each $n\in\NN$, take $\gamma_n\in\Opt(\mu_n,\nu_n)$. By similar arguments to those in the proof of proposition \ref{thm:existence-of-optimal-partial-transport}, $\{\gamma_n\}_{n\in \NN}$ is vaguely precompact, therefore, up to passing to another subsequence, we can assume that $\gamma_n\vto \gamma$ for some $\gamma\in\calM(E_\Omega)$. Moreover, we can also prove that for any compact $C\subset \Omega$, $\{(\gamma_n)_C^X\}_{n\in\NN}$ is tight and each $(\gamma_n)_C^X$ has total mass bounded above by $\mu(C)$ due to lemma \ref{lem:vague-convergence}. Therefore, $\{(\gamma_n)_C^X\}_{n\in\NN}$ is weakly convergent to $\gamma_C^X$. It follows that $\pi^1_\#\gamma|_{\Omega} = \mu$, and analogously $\pi^2_\#\gamma|_{\Omega} = \nu$. Thus $\gamma\in \Adm(\mu,\nu)$ and due to lemma \ref{lem:vague-convergence} applied to the measures $\{d(\cdot,\cdot)^p\gamma_n\}_{n\in\NN}$, which vaguely converge to $d(\cdot,\cdot)^p\gamma$, and bounded open sets $U\subset X\times X$, we get
\[
C(\gamma_U^U)\leq \liminf_{n\to\infty} C((\gamma_n)_U^U)\leq \lim_{n\to \infty} \Wb_p^p(\mu_n,\nu_n)
\]
and by the monotone convergence theorem, the claim follows.
\end{proof}

\section{Characterisation of optimal partial transport}

We now move on to study a characterisation of optimal partial transport plans analogous to theorem \ref{thm:classical criteria optimal transport}. Let us define 
\begin{align}
c(x,y) &= d(x,y)^p, \label{eq:cost function c}\\
\widetilde c(x,y) &= \min\{d(x,y)^p,d(x,A)^p+d(y,A)^p\}, \label{eq:cost function c tilde}
\end{align} 
and 
\begin{equation}\label{eq:set S}
\calS =\{(x,y)\in X\times X: \widetilde c(x,y) = c(x,y)\}.
\end{equation}

The following proposition and its proof are adaptations of \cite[Proposition 2.3]{FG10} for proper metric spaces.

\begin{prop}\label{prop:criteria for optimality}
    Let $\gamma\in \calM(E_\Omega)$ be a measure satisfying 
    \[
    \int d(x,A)^p+d(y,A)^p\ d\gamma(x,y) <\infty.
    \]
    Then the following are equivalent:
    \begin{enumerate}
        \item $\gamma\in \Opt(\pi^1_\#\gamma|_\Omega,\pi^2_\#\gamma|_\Omega)$.
        \item\label{item:optimal plans are concentrated on S} $\gamma$ is concentrated on $\calS$ and the set $\supp(\gamma)\cup A\times A$ is $\widetilde c$-cyclically monotone.
        \item there is a $c$-concave function $\phi$ such that both $\phi$ and $\phi^c$ vanish on $A$ and $\supp(\gamma)\subset \partial^c_+\phi$.
    \end{enumerate}
    Moreover, $d(x,y) = d(x,A)$ for $\gamma_\Omega^A$-a.e. $(x,y)$ and $d(x,y) = d(y,A)$ for $\gamma_A^\Omega$-a.e. $(x,y)$, whenever $\gamma$ is optimal.
\end{prop}

\begin{proof}
We start proving that (1) implies (2). We denote $\overline{\mu} = \pi^1_\#\gamma$, $\overline{\nu} = \pi^2_\#\gamma$, $\mu = \overline{\mu}|_\Omega$ and $\nu = \overline{\nu}|_\Omega$, and define
\[
\widetilde\gamma = \gamma|_{E_\Omega\cap \calS}+ (\pi^1,\proj_A\circ \pi^1)_\#\gamma|_{E_\Omega\setminus\calS} +(\proj_A\circ \pi^2,\pi^2)_\#\gamma|_{E_\Omega\setminus\calS}.
\]
It is clear that $\widetilde\gamma\in \Adm(\mu,\nu)$ and 
\[
C(\widetilde\gamma) = \int_{E_\Omega\cap\calS} d(x,y)^p\ d\gamma(x,y) + \int_{E_\Omega\setminus \calS} d(x,A)^p+d(y,A)^p\ d\gamma(x,y) \leq C(\gamma)
\]
with strict inequality if and only if $\gamma(E_\Omega\setminus\calS)>0$. Since $\gamma \in \Opt(\mu,\nu)$, we get that $C(\widetilde\gamma) = C(\gamma)$, which implies that $\gamma(E_\Omega\setminus\calS) = 0$, i.e.\ $\gamma$ is concentrated on $\calS$. In particular,
\[
C(\gamma) = \int_{E_\Omega} \widetilde c(x,y)\ d\gamma(x,y),
\]
where $\widetilde c$ is given by \eqref{eq:cost function c tilde}. Now suppose that a measure $\eta$ on $X\times X$ satisfies $\pi^1_\#\eta=\overline{\mu}$ and $\pi^2_\#\eta = \overline\nu$. Then we can define $\widetilde\eta$ in analogy to how we defined $\widetilde\gamma$. Clearly $\widetilde \eta\in \Adm(\mu,\nu)$, and 
\[
C(\widetilde\eta) = \int_{E_\Omega} \widetilde c(x,y)\ d\eta(x,y).
\]
In particular, 
\[
\int_{X\times X} \widetilde c(x,y)\ d\gamma(x,y) = C(\gamma) \leq C(\widetilde\eta) = \int_{X\times X} \widetilde c(x,y)\ d\eta(x,y)
\]
for any $\eta$ with the same marginals as $\gamma$. In other words, $\gamma$ is an optimal plan between $\overline{\mu}$ and $\overline{\nu}$ with respect to the cost function $\widetilde c$ in the usual sense. Theorem \ref{thm:classical criteria optimal transport} implies that $\supp(\gamma)$ is $\widetilde c$-cyclically monotone. Moreover, given $\{(x_i,y_i)\}_{i=1}^{n}\subset \supp(\gamma)\cup A\times A$ and $\sigma\in \Sigma_n$, we have 
\begin{align*}
\sum_{i=1}^{n}\widetilde c(x_i,y_{\sigma(i)}) 
&= 
\sum_{\substack{(x_i,y_i)\in \supp(\gamma)\\(x_{\sigma(i)},y_{\sigma(i)})\in \supp(\gamma)}} \widetilde c(x_i,y_{\sigma(i)})
+
\sum_{\substack{(x_i,y_i)\in \supp(\gamma)\\(x_{\sigma(i)},y_{\sigma(i)})\in A\times A\setminus\supp(\gamma)}} \widetilde c(x_i,y_{\sigma(i)})\\
&\quad 
+
\sum_{\substack{(x_i,y_i)\in A\times A\setminus\supp(\gamma)\\(x_{\sigma(i)},y_{\sigma(i)})\in \supp(\gamma)}} \widetilde c(x_i,y_{\sigma(i)})
+
\sum_{\substack{(x_i,y_i)\in A\times A\setminus\supp(\gamma)\\(x_{\sigma(i)},y_{\sigma(i)})\in A\times A\setminus\supp(\gamma)}} \widetilde c(x_i,y_{\sigma(i)})
\end{align*}
where it is clear that the number of summands in the second and the third summations are the same, whereas the fourth summation vanishes since $\widetilde c(x,y) = 0$ for any $(x,y)\in A\times A$. Let $\{j_i\}_{i=1}^{p}$, $\{k_i\}_{i=1}^{q}$ and $\{l_i\}_{i=1}^{q}$ be the sets of indices for the first, second and third sums, respectively, and define a permutation $\widetilde\sigma$ of the indices $\{j_1,\dots,j_p,k_1,\dots,k_q\}=\{i:(x_i,y_i)\in \supp(\gamma)\}$ by $\widetilde\sigma(j_r) = \sigma(j_r)$ for $r=1,\dots, p$ and $\widetilde\sigma(k_s) = \sigma(l_s)$ for $s=1,\dots, q$. Then
\begin{align*}
\sum_{i=1}^{n}\widetilde c(x_i,y_{\sigma(i)}) &= \sum_{i=1}^{p} \widetilde c(x_{j_i},y_{\sigma(j_i)})+ \sum_{i=1}^{q} \widetilde c(x_{k_i},y_{\sigma(k_i)}) + \widetilde c(x_{l_i},y_{\sigma(l_i)})\\
&=\sum_{i=1}^{p} \widetilde c(x_{j_i},y_{\sigma(j_i)})+ \sum_{i=1}^{q} \widetilde c(x_{k_i},y_{\sigma(k_i)}) + \widetilde c(y_{\sigma(k_i)},x_{l_i})+\widetilde c(x_{l_i},y_{\sigma(l_i)})\\
&\geq \sum_{i=1}^{p} \widetilde c(x_{j_i},y_{\sigma(j_i)})+ \sum_{i=1}^{q} \widetilde c(x_{k_i},y_{\sigma(l_i)})\\
&=\sum_{(x_i,y_i)\in \supp(\gamma)} \widetilde c(x_i,y_{\widetilde\sigma(i)})\\
&\geq \sum_{(x_i,y_i)\in \supp(\gamma)} \widetilde c(x_i,y_i)\\
&=\sum_{i=1}^{n} \widetilde c(x_i,y_i)
\end{align*}
where in the second and last equalities we have used the fact that $\widetilde c(x,y) = 0$ for any $(x,y)\in A\times A$, whereas for the first and second inequalities we have used the easily checked inequality $\widetilde c (x,y) + \widetilde c(z,y) \geq \widetilde c(x,z)$ that holds for any $x,z\in X$ and $y\in A$, and the fact that $\supp(\gamma)$ is $\widetilde c$-cyclically monotone. This shows (2).

Now, to prove that (2) implies (3), observe that, by theorem \ref{thm:classical criteria optimal transport}, the $\widetilde c$-cyclical monotonicity of $\supp(\gamma)\cup A\times A$ implies that there exists a $\widetilde c$-concave function, say $\phi$, such that $\supp(\gamma)\cup A\times A\subset \partial^{\widetilde c}_+\phi$. In particular, 
\[
\phi(x) + \phi^{\widetilde c}(y) = \widetilde c(x,y) = 0
\]
for any $(x,y)\in A\times A$, which implies that both $\phi$ and $\phi^{\widetilde c}$ are constant on $A$. Since the $\widetilde c$-concavity is invariant under addition of constants, we can assume that $\phi=\phi^{\widetilde c}=0$ on $A$.

On the other hand, we can see that $\phi$ is $c$-concave. Indeed, one can prove that, for any $y\in X$, the map $x\mapsto \widetilde c(x,y)$ is $c$-concave. Indeed, if we define $\psi_y\colon X\to \RR\cup\{-\infty\}$ by
\[
\psi_y(z) = 
\begin{cases}
-\infty & \text{if}\ z\in X\setminus (A\cup\{y\}) \\
0 & \text{if}\ z=y \\
-d(y,A)^p& \text{if}\ z\in A,
\end{cases}
\]
then it is clear that
\[
\inf_{z\in X} c(x,z)-\psi_y(z) = \min\left\{d(x,y)^p,\inf_{z\in A} d(x,z)^p+d(y,A)^p\right\}= \widetilde c(x,y),
\]
and since $\phi$ is $\widetilde c$-concave, there exists a function $\psi\colon X\to \RR\cup\{-\infty\}$ such that $\phi = \psi^{\widetilde c}$, that is,
\begin{align*}
\phi(x) &= \inf_{y\in X} \widetilde c(x,y) - \psi(y)\\
&= \inf_{y\in X} \inf_{z\in X} c(x,z) - \psi_y(z) - \psi(y)\\
&= \inf_{z\in X} \inf_{y\in X} c(x,z) - \psi_y(z)-\psi(y)\\
&= \inf_{z\in X} c(x,z) - \sup_{y\in X} \psi_y(z) + \psi(y)\\
&= \inf_{z\in X} c(x,z) - \eta(z),
\end{align*}
where $\eta(z) = \sup_{y\in X} \psi_y(z)+\psi(y)$. Observe this is a well defined function $X\to \RR\cup\{-\infty\}$ since, if $z\in A$ then 
\[
\sup_{y\in X}\psi_y(z)+\psi(y) = \sup_{y\in X}-d(y,A)^p+\psi(y) = -\inf_{y\in X}\widetilde c(z,y)-\psi(y)= -\phi(z) = 0
\]
and if $z\not\in A$, then
\[
\psi_y(z)+\psi(y) = 
\begin{cases}
    -\infty & \text{if}\ y\neq z\\
    \psi(z) & \text{if}\ y=z
\end{cases}
\]    
which implies that
$\sup_{y\in X} \psi_y(z) + \psi(y) = \psi(z)$.

Moreover, if $(x,y)\in \partial^{\widetilde c}_+\phi\cap\calS$ then 
\[
\phi^{\widetilde c}(y) = \widetilde c(x,y) - \phi(x) = c(x,y) - \phi(x) 
\]
whereas 
\[
\phi^{\widetilde c}(y) \leq \widetilde c(x',y) - \phi(x') \leq c(x',y) - \phi(x')
\]
for any $x'\in X$, which means that $\phi^{\widetilde c}(y) = \phi^c(y)$, therefore $(x,y)\in \partial^c_+ \phi$. In particular, since $\gamma$ is concentrated on $\partial^{\widetilde c}_+\phi\cap\calS$, we get that it is also concentrated on $\partial^c_+\phi$, which implies that $\supp(\gamma) \subset \partial^c_+\phi$.

On the other hand, if $x\in A$ then, since $\phi = \phi^{\widetilde c} = 0$ on $A$, we get
\[
\phi(x)+\phi^{\widetilde c}(x) = \widetilde c(x,x) = c(x,x) = 0
\]
which implies that $(x,x) \in \partial^{\widetilde c}_+\phi\cap\calS$, therefore $(x,x)\in \partial^c_+\phi$. In particular,
\[
\phi^c(x) = \phi(x) + \phi^c(x) = c(x,x) = 0
\]
therefore $\phi^c =0$ on $A$. This proves (3).

Finally, to prove that (3) implies (1), consider $\phi$ a $c$
-concave function such that $\supp(\gamma) \subset \partial^c_+\phi$ and $\phi=\phi^c = 0$ on $A$, and choose $\widetilde \gamma\in \Adm(\mu,\nu)$. Then, since $\pi^1_\#\gamma|_\Omega = \pi^1_\#\widetilde \gamma|_\Omega$, we get
\[
\int_X \phi\ d\pi^1_\#\gamma = \int_\Omega \phi\ d\pi^1_\#\gamma= \int_\Omega \phi\ d\pi^1_\#\widetilde\gamma= \int_X \phi\ d\pi^1_\#\widetilde \gamma.
\]
Analogously with $\phi^c$. Then, we get
\begin{align*}
\int_{E_\Omega} d(x,y)^p \ d\gamma(x,y) &= \int_{E_\Omega} \phi(x)+\phi^c(y)\ d\gamma(x,y)\\
&= \int_{X} \phi(x)\ d\pi^1_\#\gamma(x) + \int_{X} \phi^c(y)\ d\pi^2_\#\gamma(y)\\
&= \int_{X} \phi(x)\ d\pi^1_\#\widetilde\gamma(x) + \int_{X} \phi^c(y)\ d\pi^2_\#\widetilde\gamma(y)\\
&= \int_{E_\Omega} \phi(x) + \phi^c(y)\ d\widetilde\gamma(x,y)\\
&\leq \int_{E_\Omega} d(x,y)^p\ d\widetilde\gamma(x,y)
\end{align*}
where the last inequality is due to the inequality $\phi(x)+\phi^c(y)\leq c(x,y)$ that holds for general $(x,y)\in X\times X$. This argument implies that $\gamma \in \Opt(\mu,\nu)$.

To prove the last part of the statement, we only need to observe that, whenever $\gamma\in\Opt(\mu,\nu)$ for some $\mu,\nu\in \calM_p(X,A)$, then by (2) we have 
\[
\supp(\gamma_\Omega^A)\subset\supp(\gamma)\cap \Omega\times A\subset \calS\cap \Omega\times A
\]
and for any $(x,y)\in \calS\cap\Omega\times A$ we have
\[
d(x,A)^p\leq d(x,y)^p = \min\{d(x,y)^p,d(x,A)^p\}\leq d(x,A)^p
\]
which implies $d(x,y)=d(x,A)$ for any $(x,y)\in\supp(\gamma_\Omega^A)$. Analogously with $\gamma_A^\Omega$.
\end{proof}

\section{Properties of \texorpdfstring{$\calM_p(X,A)$}{the spaces of Radon measures}}
In this section we prove that $\calM_p(X,A)$ inherits different properties from the space $X$. Namely, we prove that $\calM_p(X,A)$ is complete, separable, geodesic and non-negatively curved in the Alexandrov sense whenever the underlying space $X$ is so. Proofs are similar to the corresponding results for the classical Wasserstein spaces of probability measures, and we include them for the sake of completeness.
\subsection{Completeness and separability}\label{sec:completeness and separability-OPT}
In this subsection we prove that $\calM_p(X,A)$ inherits the properties of completeness and separability from the underlying space $X$. This statement and its proof are adaptations of \cite[Proposition 2.7]{FG10} for proper metric spaces.
\begin{prop}\label{prop:competeness and separability}
The space $\calM_p(X,A)$ is complete and separable.    
\end{prop}
\begin{proof}
For the separability of $\calM_p(X,A)$, if we choose a countable dense set $S\subset X$ and define,
\[
F = \left\{\sum_{i\in I} q_i \delta_{x_i}: x_i\in S\cap \Omega,\ q_i\in \QQ_+,\ I\subset \NN\ \text{is finite}\right\},
\]
it is easy to check that $F$ is countable and dense in $\calM_p(X,A)$, following the same argument in \cite[Theorem 6.18]{V09}.

To prove that $\calM_p(X,A)$ is complete, we consider a Cauchy sequence $\{\mu_n\}_{n\in\NN}\subset \calM_p(X,A)$. Since $\{\Wb_p(\mu_n,0)\}_{n\in \NN}$ is a bounded sequence in $\RR$, say by some $C>0$, then for any compact $K\subset \Omega$ we have
\[
\mu_n(K)\leq \frac{1}{r^p}\int_\Omega d(x,A)^p\ d\mu_n(x) = \frac{1}{r^p}\Wb_p^p(\mu_n,0) \leq \frac{C^p}{r^p}
\]
where $d(x,A)>r$ for any $x\in K$. In particular
\[
\sup\{\mu_n(K):n\in\NN\}<\infty
\] for any compact $K\subset \Omega$, which due to lemma \ref{lem:vague-relative-compactness}, implies that $\{\mu_n : n\in\NN\}$ is vaguely precompact. Therefore it has a subsequence $\{\mu_{n_k}\}_{k\in\NN}$ vaguely convergent to some $\mu$. By the lower semi continuity of $\Wb_p$, we get that
\[
\Wb^p_p(\mu,0)\leq \liminf_{k\to \infty} \Wb^p_p(\mu_{n_k},0) < \infty,
\]
therefore $\mu\in \calM_p(X,A)$. Moreover, for any $n\in\NN$,
\[
\Wb_p(\mu_n,\mu)\leq \liminf_{k\to \infty} \Wb^p_p(\mu_n,\mu_{n_k}),
\]
which implies
\[
\lim_{n\to\infty}\Wb_p(\mu_n,\mu)\leq \lim_{n\to\infty}\liminf_{k\to \infty} \Wb^p_p(\mu_n,\mu_{n_k}) = 0
\] where the last inequality comes from the fact that $\{\mu_n\}_{n\in\NN}$ is Cauchy. Therefore $\mu_n\to \mu$ in $\calM_p(X,A)$, which implies the completeness.
\end{proof}

\subsection{Geodesics}\label{sec:geodesicity-OPT}
We now prove that $\calM_p(X,A)$ is a geodesic space whenever $X$ has this property. This is a generalisation of \cite[Proposition 2.9]{FG10} for proper metric spaces.
\begin{prop}\label{prop:geodesics}
Let $(X,A)$ be a metric pair such that $X$ is geodesic. Then $\calM_p(X,A)$ is geodesic as well. Furthermore, if $(\mu_t)_{t\in [0,1]}$ is a constant speed geodesic in $\calM_p(X,A)$, then there exists a measure $\ggamma$ on $\Geo(X)$ such that $(e_0,e_1)_\# \ggamma|_{E_\Omega} \in \Opt(\mu_0,\mu_1)$ and 
\[\mu_t = (e_t)_\#\ggamma|_{\Omega}\]
for any $t\in [0,1]$.
\end{prop}

For the proof of proposition \ref{prop:geodesics} we follow ideas from the proofs of \cite[Theorem 2.10]{AG13} and \cite[Proposition 2.9]{FG10}. In particular, we need the following technical observation.

\begin{lem}\label{lem:geodesic-selection}
There is a Borel measurable map $\GeoSel\colon X\times X\to \Geo(X)$ such that \[(e_0,e_1)\circ \GeoSel = \id_{X\times X}.\]
\end{lem}
\begin{proof}
Since $X$ and $\Geo(X)$ are complete, separable and proper, the set
\[
E = \{(x,y,\xi)\in X\times X\times \Geo(X): \xi_0=x,\ \xi_1=y,\ d(x,y)=\calL(\xi)\}
\]
is $\sigma$-compact. Moreover, by the continuity of $d$, the lower semi-continuity of the length functional $\calL$ (see \cite[Proposition 2.3.4]{BBI01}),  the continuity of the evaluation maps $e_0$, $e_1$, and the fact that $X$ is geodesic, it follows that $E$ is closed and $\pi^{1,2}(E)=X\times X$, where $\pi^{1,2}\colon X\times X\times \Geo(X)\to X\times X$ is the projection onto the first two factors. Therefore, by theorem \ref{t:azoff}, the claim follows.
\end{proof}

\begin{rem}\label{rem:intermediate points between points in S}
Observe that for any $(x,y)\in \calS$, where $\calS$ is given by \eqref{eq:set S}, and any $t\in(0,1)$ and $\xi\in\Geo(X)$ such that $(e_0,e_1)(\xi)=(x,y)$, we have $\xi_t \in \Omega$. Indeed, if this is not the case, then for some choice of $(x,y)\in\calS$, $t\in(0,1)$ and $\xi\in\Geo$ such that $(e_0,e_1)(\xi)=(x,y)$, we have $\xi_t\in A$. Therefore,
\[
d(x,A)^p+d(y,A)^p \leq d(x,\xi_t)^p+d(y,\xi_t)^p = ((1-t)^p+t^p)d(x,y)^p < d(x,y)^p
\]
which contradicts the fact that $(x,y)\in \calS$. 
\end{rem}

We also need the following lemma to prove the second part of proposition \ref{prop:geodesics}. Intuitively speaking, this lemma allows us to construct a sequence of measures on $\Geo(X)$ that interpolate a given geodesic $(\mu_t)_{t\in [0,1]}$ in $\calM_p(X,A)$ at arbitrarily fine dyadic rational parameters in $[0,1]$.

\begin{lem}\label{lem:geodesic measure on dyadics}
Let $(\mu_t)_{t\in[0,1]}$ be a geodesic in $\calM_p(X,A)$. Then, for any $m\in\NN$, we can find $\ggamma^m\in\calB(\Geo(X))$ such that 
\begin{align*}
(e_{j/2^m},e_{k/2^m})_\#\ggamma^m|_{E_\Omega} &\in \Opt(\mu_{j/2^m},\mu_{k/2^m}),\\
\supp((e_{j/2^m},e_{k/2^m})_\#\ggamma^m|_{A\times A}) &\subset \Delta(A\times A), 
\end{align*}
for any $j,k\in\{0,\dots,2^m\}$.
\end{lem}

\begin{proof}
For each $m\in\NN$ and each $i\in\{0,\dots, 2^m-1\}$, choose 
\[
\gamma^{i,m}\in \Opt(\mu_{i/2^m},\mu_{(i+1)/2^m})
\]
and apply lemma \ref{lem:gluing} iteratively to get $\gamma^m\in\calB(X^{2^m+1})$ such that
\[
\pi^{i,i+1}_\#\gamma^m = \gamma^{i,m}+\sigma^{i,m},
\]
where $\sigma^{i,m}$ is supported on $\Delta(A\times A)$. Then, by letting $G^m \colon X^{2^m+1}\to \calC([0,1],X)$ be the Borel measurable map given by
\[
(x_0,\dots,x_{2^m}) \mapsto \GeoSel(x_0,x_1)*\dots *\GeoSel(x_{2^m},x_{2^m+1}),
\]
where $*$ denotes concatenation of paths, we can define
\[
\ggamma^m = G^m_\#\gamma^m.
\]
Now, for any $j,k\in\{0,\dots,2^m\}$, we have
\begin{align*}
\|d(e_{j/2^m},e_{k/2^m})\|_{L^p(\ggamma^m)}
&\leq \left\|\sum_{i=j}^{k-1}d(e_{i/2^m},e_{(i+1)/2^m})\right\|_{L^p(\ggamma^m)}\\
&\leq \sum_{i=j}^{k-1}\|d(e_{i/2^m},e_{(i+1)/2^m})\|_{L^p(\ggamma^m)}\\
&= \sum_{i=j}^{k-1}\Wb_p(\mu_{i/2^m},\mu_{(i+1)/2^m})\\
&= \Wb_p(\mu_{j/2^m},\mu_{k/2^m})
\end{align*}
where the first inequality is given by the monotonicity of the integral; the second one is the triangle inequality in $L^p(\ggamma^m)$, and the last two lines are consequences of the definition of $\ggamma^m$ and the fact that $(\mu_t)_{t\in [0,1]}$ is a geodesic. As a consequence, we get that 
\[
(e_{j/2^m},e_{k/2^m})_\#\ggamma^m|_{E_\Omega} \in \Opt(\mu_{j/2^m},\mu_{k/2^m})
\]
and
\[
d(\xi_{j/2^m},\xi_{k/2^m}) = \sum_{i=j}^{k-1}d(\xi_{i/2^m},\xi_{(i+1)/2^m})
\]
for $\ggamma^m$-a.e. $\xi\in\calC([0,1],X)$, and the claim follows.
\end{proof}

The next results are similar in spirit to lemma \ref{lem:weak convergence of minimising sequence} and Corollary \ref{lem:gamma-eps is closed}, and they will allow us to construct the limit measure we need for the second part of proposition \ref{prop:geodesics}. Let us define
\begin{equation}\label{eq:optgeo}
\OptGeo(\mu,\nu) = \{\ggamma\in\calM((e_0,e_1)^{-1}(E_\Omega)): (e_0,e_1)_\#\ggamma\in\Opt(\mu,\nu)\}.
\end{equation}

\begin{lem}\label{lem:weak precompactness of geodesic measures}
For any $\mu,\nu,\rho\in\calM_p(X,A)$, any compact $C\subset \Omega$, and any $t\in [0,1]$, the set 
\begin{equation}
\left\{\ggamma|_{e^{-1}_t(C)}:\ggamma\in\OptGeo(\mu,\nu),\ (e_t)_\#\ggamma|_\Omega=\rho\right\}
\end{equation}
is weakly relatively compact. 
\end{lem}
\begin{proof}
Let $p_0\in X$ and $R>r>0$ such that $C\subset \overline{B}_r(p_0)\subset \overline{B}_R(p_0)$. Then, an argument analogous to the proof of lemma \ref{lem:weak convergence of minimising sequence} shows that
\begin{equation}\label{eq:tightness of ggammas}
\ggamma(e^{-1}_t(C)\setminus (e_0,e_1)^{-1}(\overline{B}_R(p_0)\times \overline{B}_R(p_0)))\leq
\dfrac{\Wb_p^p(\mu_0,\mu_1)}{(R-r)^p}
\end{equation}
for any $\ggamma\in\OptGeo(\mu,\nu)$. Observe that $(e_0,e_1)^{-1}(\overline{B}_R(p_0)\times \overline{B}_R(p_0))$ is closed and, being the set of geodesics with endpoints in $\overline{B}_R(p_0)$, is contained in $\Geo(\overline{B}_{2R}(X))$, which is compact due to Arzel\`a--Ascoli theorem and the fact that $X$ is proper. By fixing $r>0$ and letting $R$ tend to infinity, the right hand side of inequality \eqref{eq:tightness of ggammas} can be made arbitrarily small, uniformly over $\ggamma$, which implies tightness.

On the other hand, each $\ggamma|_{e^{-1}_t(C)}$ has total mass $\rho(C)<\infty$. Therefore we have uniformly bounded total variation, and the claim follows from lemma \ref{lem:prokhorov}.
\end{proof}

\begin{cor}\label{cor:optgeo is vaguely closed}
For any $\mu,\nu,\rho\in\calM_p(X,A)$ and any $t\in [0,1]$, the set 
\begin{equation}
\left\{\ggamma:\ggamma\in\OptGeo(\mu,\nu),\ (e_t)_\#\ggamma|_\Omega=\rho\right\}
\end{equation}
is vaguely closed in $\calM((e_0,e_1)^{-1}(E_\Omega))$. 
\end{cor}
\begin{proof}
Let $\{\ggamma^m\}_{m\in\NN}\subset\OptGeo(\mu,\nu)$ be such that $(e_t)_\#\ggamma^m|_\Omega = \rho$ for all $m\in\NN$, and assume that $\ggamma^m\vto \ggamma$ for some $\ggamma\in\calM((e_0,e_1)^{-1}(E_\Omega))$. Thanks to lemma \ref{lem:weak precompactness of geodesic measures}, for any compact $C\subset \Omega$, $\{\ggamma^m|_{e^{-1}_t(C)}\}_{m\in\NN}$ is weakly relatively compact, therefore $\ggamma^m|_{e^{-1}_t(C)}\wto \ggamma|_{e^{-1}_t(C)}$. In particular, if $f\in\calC_c(\Omega)$ then
\begin{align*}
\int_{\Omega} f\ d(e_t)_\#\ggamma &= \int_{e_t^{-1}(\supp(f))} f\circ e_t\  d\ggamma \\
&= \lim_{m\to\infty} \int_{e_t^{-1}(\supp(f))} f\circ e_t\  d\ggamma^m\\
&= \lim_{m\to\infty} \int_{\supp(f)} f\ d(e_t)_\#\ggamma^m\\
&= \int_{\Omega} f\ d\mu_t,
\end{align*}
which implies $(e_t)_\#\ggamma|_{\Omega}=\rho$. Analogously, $(e_0,e_1)_\#\ggamma\in\Adm(\mu,\nu)$.

Now, we prove that $\ggamma\in\OptGeo(\mu,\nu)$. Indeed, by lemma \ref{lem:vague-convergence} applied to the sequence $\{d(e_0,e_1)^p \ggamma^{m}\}_{m\in\NN}$, which is vaguely convergent to $d(e_0,e_1)^p\ggamma$, and any bounded open set $U\subset (e_0,e_1)^{-1}(E_\Omega)$, we get
    \[
    C((e_0,e_1)_\#\ggamma|_U) \leq \liminf_{m\to \infty} C((e_0,e_1)_\#\ggamma^m|_U)\leq \Wb_p^p(\mu,\nu).
    \]
By the monotone convergence theorem, the claim follows.
\end{proof}

\begin{proof}[Proof of proposition \ref{prop:geodesics}]
Let $\mu^0,\mu^1\in \calM_p(X,A)$ and choose $\gamma \in \Opt(\mu^0,\mu^1)$. By lemma \ref{lem:geodesic-selection} there is a measurable map $\GeoSel\colon X\times X\to \Geo(X)$ such that $\GeoSel(x,y)$ is a constant speed geodesic joining $x$ and $y$. We define a measure $\ggamma$ on $\Geo(X)$ by
\[
\ggamma = \GeoSel_\# \gamma.
\]
In particular, since $\gamma$ is concentrated on $\calS$ by item \ref{item:optimal plans are concentrated on S} in proposition \ref{prop:criteria for optimality}, then $\ggamma$ is concentrated on $(e_0,e_1)^{-1}(\calS)$, which implies that $(e_t)_\#\ggamma$ is concentrated on $\Omega$, for any $t\in (0,1)$, by remark \ref{rem:intermediate points between points in S}. Observe, however, that this is not necessarily the case for $t=0$ and $t=1$. Let $\mu_t$ be given by 
\[
\mu_t = (e_t)_\#\ggamma|_\Omega.
\]
We claim that the curve $(\mu_t)_{t\in [0,1]}$ is a constant speed geodesic joining $\mu^0$ and $\mu^1$ in $\calM_p(X,A)$. Indeed, since $(e_0,e_t)\circ \GeoSel = \id$ then 
\[
\mu^0=\pi^1_\#\gamma|_\Omega=\pi^1_\#(e_0,e_1)_\#\GeoSel_\#\gamma|_\Omega=(e_0)_\#\ggamma|_\Omega=\mu_0
\]
and, analogously, $\mu^1 = \mu_1$. Moreover, for any $s,t\in [0,1]$,
\begin{align*}
    \Wb_p^p(\mu_t,\mu_s) &= \Wb_p^p({(e_t)}_\#\ggamma|_\Omega,{(e_s)}_\#\ggamma|_\Omega) \\
    &\leq C((e_t,e_s)_\#\ggamma|_{E_\Omega})\\
    &=\int_{E_\Omega} d(x,y)^p\ d(e_t,e_s)_\#\ggamma(x,y)\\
    &=\int_{E_\Omega} d(\GeoSel(x,y)_t,\GeoSel(x,y)_s)^p\ d\gamma(x,y)\\
    &=|t-s|^p\int_{E_\Omega} d(x,y)^p\ d\gamma(x,y)\\
    &=|t-s|^p \Wb_p^p(\mu_0,\mu_1).
\end{align*}
This argument implies both that $\mu_t\in \calM_p(X,A)$ for any $t\in [0,1]$, by the triangle inequality, and that $(\mu_t)_{t\in [0,1]}\in \Geo(\calM_p(X,A))$.

Now, for the second part of the theorem, let $(\mu_t)_{t\in [0,1]}$ be a geodesic in $\calM_p(X,A)$. We want to construct $\ggamma\in\calM((e_0,e_1)^{-1}(E_\Omega))$ such that $(e_0,e_1)_\#\ggamma\in\Opt(\mu_0,\mu_1)$ and $(e_t)_\#\ggamma|_\Omega = \mu_t$ for all $t\in [0,1]$. We will get such $\ggamma$ as a limit of a sequence of measures given by lemma \ref{lem:geodesic measure on dyadics}. 

Indeed, for each $m\in\NN$, let $\ggamma^m\in\calB(\Geo(X))$ be as in lemma \ref{lem:geodesic measure on dyadics}. In particular, $(e_0,e_1)_\#\ggamma^m|_{E_\Omega}\in\Opt(\mu_0,\mu_1)$. Due to item \ref{item:optimal plans are concentrated on S} in proposition \ref{prop:criteria for optimality}, $\ggamma^m|_{(e_0,e_1)^{-1}(E_\Omega)}$ is concentrated on $(e_0,e_1)^{-1}(\calS\cap E_\Omega)$. Additionally, lemma \ref{lem:geodesic measure on dyadics} implies that $\ggamma^m|_{(e_0,e_1)^{-1}(A\times A)}$ is supported on constant geodesics. Therefore, without loss of generality, we can assume that $\ggamma^m=\ggamma^m|_{(e_0,e_1)^{-1}(E_\Omega)}$.

We now prove that $\{\ggamma^m\}_{m\in\NN}$ is vaguely relatively compact in $\calM((e_0,e_1)^{-1}(E_\Omega))$. Indeed, if $\calK\subset (e_0,e_1)^{-1}(\calS\cap E_\Omega)$ is a compact set, then $e_{1/2}(\calK)\subset\Omega$ by remark \ref{rem:intermediate points between points in S}, and thanks to lemma \ref{lem:geodesic measure on dyadics},
\[
 \ggamma^m(\calK) \leq \ggamma^m((e_{1/2})^{-1}(e_{1/2}(\calK))) = (e_{1/2})_\#\ggamma^m(e_{1/2}(\calK)) = \mu_{1/2}(e_{1/2}(\calK)).
\]
Since $\mu_{1/2}$ is a Radon measure on $\Omega$, and $e_{1/2}(\calK)$ is compact, we get that
\[
 \sup_{m\in \NN}\ggamma^m(\calK) <\infty,
\]
which proves the claim, due to remark \ref{rem:locally finite equals radon} and lemma \ref{lem:vague-relative-compactness}, since $(e_0,e_1)^{-1}(E_\Omega)$ is an open subset of a separable, locally compact metric space.

Thus, we can assume, up to passing to a subsequence, that $\{\ggamma^m\}_{m\in\NN}$ is vaguely convergent to some $\ggamma\in\calM((e_0,e_1)^{-1}(E_\Omega))$. By corollary \ref{cor:optgeo is vaguely closed}, $(e_0,e_1)_\#\ggamma\in\Opt(\mu_0,\mu_1)$ and $(e_t)_\#\ggamma|_\Omega = \mu_t$ for any dyadic rational $t\in[0,1]$.

Finally, for any other $t\in [0,1]$, let $\{t_k\}_{k\in\NN}$ and $\{t^l\}_{l\in\NN}$ be two sequences of dyadic rational numbers converging to $t$, with $t_k\leq t\leq t^l$ for any $k,l\in\NN$, and observe that $(e_{t_k},e_t)_\#\ggamma|_{E_\Omega}\in \Adm(\mu_{t_k},(e_t)_\#\ggamma|_\Omega)$. Therefore, 
\begin{align*}
\Wb_p^p(\mu_{t_k},(e_t)_\#\ggamma|_\Omega) &\leq C((e_{t_k},e_t)_\#\ggamma|_{E_\Omega})\\
&\leq C((e_{t_k},e_{t^l})_\#\ggamma|_{E_\Omega})\\
&= \Wb_p^p(\mu_{t_k},\mu_{t^l}).
\end{align*}
By letting $k,l\to\infty$, we get that $(e_t)_\#\ggamma|_\Omega = \mu_t$ as claimed.
\end{proof}

We now prove that $\calM_p(X,A)$ inherits the property of being non-branching, whenever $p>1$. This is analogous to the second half of \cite[Proposition 2.9]{FG10}, and the proof adapts ideas from \cite[Proposition 2.16]{AG13}.

\begin{prop}\label{prop:non-branching}
Let $(X,A)$ be a metric pair such that $X$ is geodesic and non-branching, and $p\in (1,\infty)$. Then $\calM_p(X,A)$ is non branching. Furthermore, if $(\mu_t)_{t\in [0,1]} \in \Geo(\calM_p(X,A))$, then for any $t\in (0,1)$ and any $\gamma\in\Opt(\mu_0,\mu_{t})$, $\gamma^\Omega_X$ is unique and it is induced by a map.
\end{prop}
\begin{proof}
Let $(\mu_t)_{t\in [0,1]}\in\Geo(\calM_p(X,A))$, $t\in (0,1)$ and consider $\gamma^1\in \Opt(\mu_0,\mu_{t})$ and $\gamma^2\in \Opt(\mu_{t},\mu_1)$. By the proof of proposition \ref{prop:geodesics}, there is measure $\gamma$ on $X^3$ such that
\[
\pi^{12}_\#\gamma = \gamma^1+\sigma^{1}
\quad\text{and}\quad 
\pi^{23}_\#\gamma = \gamma^2+\sigma^{2}
\]
for some $\sigma^1,\sigma^2$ supported on $\Delta(A\times A)$, and such that 
$\pi^{13}_\#\gamma\in \Opt(\mu_0,\mu_1)$. Moreover, 
\[
d(x,y) = td(x,z),\ d(y,z) = (1-t)d(x,z),
\]
which implies that $x,\ y,\ z$ lie in a geodesic, for $\gamma$-a.e. $(x,y,z)$.

Now consider $(x,y,z),(x',y,z')\in \supp(\gamma)$. By proposition \ref{prop:criteria for optimality}, we know that $\supp(\pi^{13}_\#\gamma)$ is in the superdifferential of a $c$-concave function  (where $c$ is given by \eqref{eq:cost function c}), which implies it is $c$-cyclically monotone. Therefore, 
\begin{align*}
    d(x,z)^p+d(x',z')^p&\leq d(x,z')^p+d(x',z)^p\\
    &\leq (d(x,y)+d(y,z'))^p+(d(x',y)+d(y,z))^p\\
    &= (td(x,z) + (1-t)d(x',z'))^p+(td(x',z')+(1-t)d(x,z))^p\\
    &\leq td(x,z)^p+(1-t)d(x',z')^p+td(x',z')^p+(1-t)d(x,z)^p\\
    &= d(x,z)^p+d(x',z')^p
\end{align*}
where the last inequality is due to the convexity of $t\mapsto t^p$ for $p>1$. Moreover, the strict convexity of the same function, and the fact the all the inequalities above are equations, imply that $d(x,z) = d(x',z')$ and 
\[
d(x,y)+d(y,z') = d(x,z').
\]
In particular, $x,\ y,\ z'$ lie in a geodesic. Since $X$ is non-branching, we get that $z=z'$, and analogously we get that $x=x'$. In other words, the map $\pi^2\colon (x,y,z)\mapsto y$ is injective in $\supp(\gamma)$. In particular, if $T$ is the inverse of $\pi^2|_{\supp(\gamma)}$, we get that 
\[
(\pi^1\circ T,\id)_\#\mu_{t} = (\gamma^1)_{X}^{\Omega}\quad \text{and}\quad 
(\id,\pi^3\circ T)_\#\mu_{t} = (\gamma^2)_{\Omega}^{X}.
\]
Therefore, $(\gamma^1)_{X}^{\Omega}$ and $(\gamma^2)_{\Omega}^{X}$ are induced by maps. Moreover, this also implies that $(\gamma^1)_{X}^{\Omega}$ is unique, because otherwise we could construct $\gamma \in \Opt(\mu_0,\mu_t)$ such that
\[
\gamma_X^\Omega = \frac{1}{2}\left((\pi^1\circ T,\id)_\#\mu_t+(\pi^1\circ T',\id)_\#\mu_t\right),
\]
which would not be induced by a map.

Finally, to prove that $\calM_p(X,A)$ is non-branching, consider geodesics $(\mu_t)_{t\in [0,1]}$, $(\mu'_t)_{t\in[0,1]}$ and $t_0\in (0,1)$ such that $\mu_0= \mu_0'$ and $\mu_{t_0}=\mu_{t_0}'$. Let $\ggamma,\ggamma'\in \calM((e_0,e_1)^{-1}(E_\Omega))$ be such that $\mu_t = (e_t)_\#\ggamma|_\Omega$ and $\mu'_t = (e_t)_\#\ggamma'|_\Omega$ for all $t\in [0,1]$. Then we have
\[
(e_0,e_{t_0})_\#\ggamma,\ (e_0,e_{t_0})_\#\ggamma'\in \Opt(\mu_0,\mu_{t_0}),
\]
and due to our previous arguments, it follows that
\[
(e_0,e_{t_0})_\#\ggamma = (e_0,e_{t_0})_\#\ggamma|_{X\times \Omega} =\ (e_0,e_{t_0})_\#\ggamma'|_{X\times \Omega} = (e_0,e_{t_0})_\#\ggamma',
\]
where we have used that $\supp((e_{t_0})_\#\ggamma)\subset \Omega$ for $t_0\in (0,1)$, thanks to remark \ref{rem:intermediate points between points in S}. Since $X$ is non-branching, the map $(e_0,e_{t_0})\colon \Geo(X)\to X\times X$ is injective, which implies
\[
\ggamma=\ggamma',
\]
therefore 
\[
\mu_1 = (e_1)_\#\ggamma|_\Omega = (e_1)_\#\ggamma'|_\Omega = \mu_1',  
\] and the proposition follows. 
\end{proof}

\subsection{Non-negative curvature}\label{sec:non-negative curvature-OPT}

In this subsection we prove that $\calM_2(X,A)$ inherits the property of having non-negative curvature in the sense of Alexandrov. This provides a new way to construct Alexandrov spaces. The proof is an adaptation of that of \cite[Theorem 2.20]{AG13}.

\begin{thm}\label{t:alexandrov-OPT}
    Let $(X,A)$ be a metric pair such that $X$ is a non-negatively curved Alexandrov space. Then $\calM_2(X,A)$ is a non-negatively curved Alexandrov space too.
\end{thm}

\begin{proof}
Since $X$ is complete, separable, proper and geodesic, it follows that $\calM_2(X,A)$ is complete and geodesic, due to theorems \ref{prop:competeness and separability} and \ref{prop:geodesics}. Now, let $(\mu_t)_{t\in [0,1]}$ be a constant speed geodesic in $\calM_2(X,A)$. Let also $\nu\in \calM_2(X,A)$ be some measure. By theorem \ref{prop:geodesics}, we know there exists a measure $\ggamma$ on $\Geo(X)$ such that $(e_0,e_1)_\#\ggamma \in \Opt(\mu_0,\mu_1)$ and $(e_t)_\#\ggamma|_\Omega = \mu_t$ for all $t\in [0,1]$. Fix $t\in (0,1)$ and consider $\gamma\in \Opt(\mu_t,\nu)$. By observing that $(e_t)_\#\ggamma =\mu_{t} = \pi^1_\#\gamma|_\Omega$ and applying the gluing lemma (theorem \ref{thm:classical gluing}), we get a measure $\aalpha\in \calM(\Geo(X)\times X)$ such that
\begin{align*}
\pi^{\Geo(X)}_\#\aalpha = \ggamma\\
(e_t\circ\pi^{\Geo(X)},\pi^X)_\#\aalpha = \gamma_\Omega^X
\end{align*}
In it therefore easy to check that 
\begin{align*}
(e_0\circ \pi^{\Geo(X)},\pi^X)_\#\aalpha|_{E_\Omega} + \gamma_A^\Omega\in \Adm(\mu_0,\nu),\\ 
(e_1\circ \pi^{\Geo(X)},\pi^X)_\#\aalpha|_{E_\Omega} + \gamma_A^\Omega \in \Adm(\mu_1,\nu)    
\end{align*}
In particular,
\begin{align*}
\Wb_2^2(\mu_t,\nu) &= \int_{E_\Omega} d(x,z)^2\ d\gamma(x,z)\\
&= \int_{\Omega\times X} d(x,z)^2\ d\gamma(x,z) + \int_{A\times \Omega} d(x,z)^2\ d\gamma(x,z)\\
&= \int_{\Geo(X)\times X}d(\xi_t,z)^2\ d\aalpha(\xi,z) + \int_{A\times \Omega} d(x,z)^2\ d\gamma(x,z)\\
&\geq \int_{\Geo(X)\times X} (1-t)d(\xi_0,z)^2 + td(\xi_1,z)^2 - (1-t)td(\xi_0,\xi_1)^2\ d\aalpha(\xi,z)\\
&\qquad + \int_{A\times \Omega} d(x,z)^2\ d\gamma(x,z)\\
&\geq (1-t)\left(\int_{E_\Omega} d(x,z)^2\ d(e_0\circ \pi^{\Geo(X)},\pi^X)_\#\aalpha(x,z) + \int_{A\times \Omega} d(x,z)^2\ d\gamma(x,z) \right)\\ 
&\qquad + t\left(\int_{E_\Omega} d(x,z)^2\ d(e_1\circ \pi^{\Geo(X)},\pi^X)_\#\aalpha(x,z) + \int_{A\times \Omega} d(x,z)^2\ d\gamma(x,z)\right)\\ 
&\qquad - (1-t)t\int_{E_\Omega} d(x,z)^2 \ d((e_0,e_1)\circ \pi^{\Geo(X)})_\#\aalpha(x,z)\\
&\geq (1-t)\Wb_2^2(\mu_0,\nu)+t\Wb_2^2(\mu_1,\nu)-(1-t)t\Wb_2^2(\mu_0,\mu_1),
\end{align*}
which proves the claim.
\end{proof}

An interesting fact about the geometric structure of $\calM_2(X,A)$ when $X$ is a non-negatively curved Alexandrov space is that the zero measure is always an {\em extremal point}, i.e.\ a point at which the space of directions has diameter bounded above by $\pi/2$ (see \cite{P07} for a more general exposition about extremal points and extremal sets in Alexandrov spaces).

\begin{prop}
    The space of directions at the zero measure, $\Sigma_0(\calM_2(X,A))$, has diameter no greater than $\pi/2$.
\end{prop}

\begin{proof}
Let $\mu,\nu\in\calM_2(X,A)$. Then, we know that
\[
\Wb_2(\mu,0)^2+\Wb_2(\nu,0)^2\geq \Wb_2(\mu,\nu)^2
\]
since the transport plan $(\id,\proj_A)_\#\mu+(\proj_A,\id)_\#\nu\in\Adm(\mu,\nu)$ is suboptimal. Therefore, if $\xi_1,\xi_2\in\Geo(\calM_2(X,A))$ are geodesics with $\xi_1(0)=\xi_2(0)=0$, then
\[
\cos\angle (\xi_1,\xi_2) = \lim_{s,t\to 0} \frac{\Wb_2(\xi_1(s),0)^2+\Wb_2(\xi_2(t),0)^2-\Wb_2(\xi_1(s),\xi_2(t))^2}{2\Wb_2(\xi_1(s),0)\Wb_2(\xi_2(t),0)}\geq 0
\]
which implies that $\angle (\xi_1,\xi_2)\leq \pi/2$.
\end{proof}

\section{Embedding of \texorpdfstring{$\calD_p(X,A)$}{spaces of persistence diagrams} into \texorpdfstring{$\calM_p(X,A)$}{spaces of positive measures}}\label{sec:embedding-OPT}

One of the motivations for studying $\calM_p(X,A)$ is that it admits a natural embedding of the space of generalised persistence diagrams $\calD_p(X,A)$, as defined in \cite{CGGGMS2022}. Indeed, we have the natural inclusion
\begin{align*}
\calD_p(X,A) &\longrightarrow \calM_p(X,A)\\
\sigma &\longmapsto \sum_{x\in \sigma|_\Omega} \delta_x.
\end{align*}
The proofs of the following results are adaptations of those of \cite[Lemma 3.4 and Proposition 3.5]{DL21}.

\begin{prop}\label{prop:approximation.measures.diagrams.with.finite.case}
Let $\mu\in \calM_p(X,A)$,  $r > 0$ and $A_r = \{x \in X: d(x, A) \leq r \}$. Let $\mu^r=\mu|_{X\setminus A_r}$. Then $\Wb_p(\mu^r,\mu) \to 0$ when $r \to 0$. Similarly, if $\sigma \in \calD_p(X,A)$, we have $d_p(\sigma^r, \sigma) \to 0$ as $r\to 0$.
\end{prop}
\begin{proof}
Let $\gamma \in \Adm(\mu,\mu^r)$ be the transport plan given by
\[
\gamma = (\id,\id)_\# \mu|_{X\setminus A_r} + (\id,\proj_A)_\#\mu|_{A_r}.
\] 
Therefore,
\[
\Wb_p^p(\mu,\mu^r) \leq \int_{A_r} d(x,A)^p\ d\mu(x).
\]
By the monotone convergence theorem applied to $\mu$ with the functions $f_r(x)=d(x,A)^p\cdot 1_{X\setminus A_r}(x)$, we conclude that $\Wb_p(\mu,\mu^r) \to 0$ as $r \to 0$. Similar arguments show that $d_p(\sigma,\sigma^r) \to 0$ as $r \to 0$.
\end{proof}

\begin{thm}
For $\sigma,\tau \in \calD_p(X,A)$, $\Wb_p(\sigma, \tau) = d_p(\sigma, \tau)$.
\end{thm}
\begin{proof}
First we consider the case when both $\sigma\setminus A$ and $\tau\setminus A$ have finite cardinality (counting multiplicity), that is, $\sigma\setminus A =\mset{x_1,\dots,x_m}$ and $\tau\setminus A = \mset{y_1,\dots,y_n}$ for some $x_i,y_j\in \Omega$, $i=1,\dots,m$, $j=1,\dots,n$. Let us define 
\begin{align*}
\widetilde\sigma &= \mset{x_1,\dots,x_m,\proj_A(y_1),\dots,\proj_A(y_n)},\\
\widetilde\tau &= \mset{y_1,\dots,y_n,\proj_A(x_1),\dots,\proj_A(x_m)}.
\end{align*}
Then, it is clear that
\[
d^p_p(\sigma,\tau) = d^p_p(\widetilde\sigma,\widetilde\tau)= \min_{P} \langle P,C\rangle_{\mathrm{HS}},
\]
where $P$ runs over all permutation matrices of size $(m+n)\times (m+n)$, $\langle\cdot,\cdot\rangle_{\mathrm{HS}}$ denotes the Hilbert--Schmidt inner product of square matrices, and 
\[
C_{ij} = 
\begin{cases}
d(x_i,y_j)^p & \text{if}\ 1\leq i\leq m,\ 1\leq j\leq n \\
d(x_i,p_A(x_{j-n})^p & \text{if}\ 1\leq i\leq m,\ n<j\leq m+n\\
d(y_j,p_A(y_{i-m}))^p & \text{if}\ m<i\leq m+n,\ 1\leq j \leq n\\
0 & \text{if}\ m<i\leq m+n,\ n<j\leq m+n
\end{cases}.
\]
Similarly, it is clear that
\begin{equation}\label{eq:embedding.Dp.Wbp.finite.case}
\Wb_p^p(\sigma,\tau) = \min_{M} \langle M,C\rangle_{\mathrm{HS}}
\end{equation}
where $M$ runs over all matrices of size $(m+n)\times (m+n)$ such that $M_{ij} \geq 0$ and $\sum_{i=1}^{m+n} M_{ij}=\sum_{j=1}^{m+n} M_{ij} = 1$. 

However, it is known that minimisers in equation \eqref{eq:embedding.Dp.Wbp.finite.case} are permutation matrices (see \cite{B46,S15}). This proves the finite case.

For arbitrary $\sigma,\tau\in\calD_p(X,A)$, consider $r>0$ and observe that both $\sigma^r$ and $\tau^r$ contain finitely many points in $\Omega$. Then, due to proposition \ref{prop:approximation.measures.diagrams.with.finite.case}, we get that
\[
\Wb_p(\sigma,\tau) = \lim_{r\to 0} \Wb_p(\sigma^r,\tau^r) = \lim_{r\to 0} d_p(\sigma^r,\tau^r) = d_p(\sigma,\tau).\qedhere
\]
\end{proof}

\bibliographystyle{amsplain}
\bibliography{references}

\end{document}